\documentclass[11pt, a4paper]{amsart}
\numberwithin{equation}{section}
\usepackage{amsmath}
\usepackage{amsthm}
\usepackage{amsfonts}
\usepackage{amssymb}
\usepackage{verbatim}
\usepackage{graphicx}
\usepackage{geometry}
\usepackage{a4wide}
\usepackage[latin1]{inputenc}

\DeclareMathOperator{\CT}{CT}

\DeclareMathOperator{\Gr}{Gr}

\DeclareMathOperator{\SL}{SL}
\DeclareMathOperator{\Mp}{Mp}

\renewcommand{\O}{\mathrm{O}}

\newcommand{\N}{\mathbb{N}}
\newcommand{\Z}{\mathbb{Z}}
\newcommand{\R}{\mathbb{R}}
\newcommand{\C}{\mathbb{C}}
\newcommand{\Q}{\mathbb{Q}}
\renewcommand{\H}{\mathbb{H}}

\DeclareMathOperator{\res}{res}
\DeclareMathOperator{\tr}{tr}

\DeclareMathOperator{\e}{\mathfrak{e}}

\DeclareMathOperator{\cusp}{cusp}

\DeclareMathOperator{\spt}{spt}
\DeclareMathOperator{\sgn}{sgn}
\DeclareMathOperator{\hol}{hol}

	\newtheorem{Satz}{Satz}[section]
	\newtheorem{Theorem}[Satz]{Theorem}
	
	\newtheorem{Proposition}[Satz]{Proposition} 
	\newtheorem{Corollary}[Satz]{Corollary}

	\theoremstyle{definition} 
	\newtheorem{Definition}[Satz]{Definition}
	\newtheorem{Example}[Satz]{Example}
	\newtheorem{Remark}[Satz]{Remark}
\date{\today}

\author{Jan Hendrik Bruinier}
\address{Fachbereich Mathematik, Technische Universit\"at Darmstadt, Schlossgartenstra{\ss}e 7, D--64289 Darmstadt, Germany}
\email{bruinier@mathematik.tu-darmstadt.de}

\author{Markus Schwagenscheidt}
\address{Mathematical Institute, University of Cologne, Weyertal 86-90, D--50931 Cologne, Germany}
\email{mschwage@math.uni-koeln.de}

\title[Theta lifts for Lorentzian lattices]{Theta lifts for Lorentzian lattices and coefficients of mock theta functions}

\thanks{The first author is partially supported by the DFG Research Unit FOR 1920 \lq Symmetry, Geometry and Arithmetic\rq \, and the LOEWE Reseach Unit USAG. The second author is supported by the SFB-TRR 191 \lq Symplectic Structures in Geometry, Algebra and Dynamics\rq, funded by the DFG. 
The authors thank Michael Mertens for helpful discussions}

\begin{document} 

\begin{abstract}
	We evaluate regularized theta lifts for Lorentzian lattices in three different ways. In particular, we obtain formulas for their values at special points involving coefficients of mock theta functions. By comparing the different evaluations, we derive recurrences for the coefficients of mock theta functions, such as Hurwitz class numbers, Andrews' spt-function, and Ramanujan's mock theta functions.
\end{abstract}

\maketitle

\section{Introduction}

Special values of modular forms at complex multiplication points play a prominent role in number theory. For instance, they are crucial in explicit class field theory and in the Chowla-Selberg formula. Moreover, the evaluation of automorphic Green functions at CM points determines archimedean height pairings appearing in the Gross-Zagier formula and some of its generalizations \cite{gz}. These can be computed using suitable see-saw dual pairs for regularized theta lifts of harmonic Maass forms to orthogonal groups of signature $(2,n)$, see \cite{schofer} or \cite{bruinieryang}, for example.

In the present paper, we study similar special values of regularized theta lifts 
for Lorentzian lattices, that is, for even lattices of signature $(1,n)$. It was shown in \cite{bruinierzemel} that these values appear naturally as vanishing orders of Borcherds products along boundary divisors of toroidal compactifications of orthogonal Shimura varieties. Here we compute the special values in several different ways. As an application, we derive general recursive formulas for the coefficients of mock theta functions. These include the Hurwitz-Kronecker class number relations and some generalizations as special cases. We give a more detailed outline of our results now.
 

\subsection{Theta lifts for Lorentzian lattices} 

Let $L$ be 
an even lattice of signature $(1,n)$ with $n \geq 1$. Let $L'$ denote its dual lattice and let $\Gr(L)$ be the Grassmannian of positive lines in $L \otimes \R$. The associated Siegel theta function $\Theta_{L}(\tau,z)$ is defined on $\H \times \Gr(L)$ and takes values in the group ring $\C[L'/L]$. As a function of $\tau \in \H$ it transforms like a modular form of weight $(1-n)/2$ for the Weil representation $\rho_{L}$ associated to $L$. Furthermore, it is invariant in $z$ under the subgroup of $\O(L)$ fixing $L'/L$. Following Borcherds \cite{borcherds}, for a harmonic Maass form $f$ of weight $(1-n)/2$ for $\rho_{L}$ we define the \emph{regularized theta lift}
\begin{align*}
\Phi(f,z) = \lim_{T \to \infty}\int_{\mathcal{F}_{T}}\langle f(\tau),\Theta_{L}(\tau,z)\rangle v^{\frac{1-n}{2}}\frac{du dv}{v^{2}}, \qquad (\tau = u+iv \in \H, z \in \Gr(L)),
\end{align*}
where $\mathcal{F}_{T}$ denotes the standard fundamental domain for $\SL_{2}(\Z)$, cut off at height $v = T$, and $\langle \cdot,\cdot \rangle$ is the natural inner product on $\C[L'/L]$ which is antilinear in the second variable. One of the special features in signature $(1,n)$ is that the theta lift $\Phi(f,z)$ converges without the additional regularization employed by Borcherds in \cite{borcherds}. Another peculiarity is the fact that $\Phi(f,z)$ defines a \emph{continuous} function on $\Gr(L)$ which is real-analytic up to singularities along certain sub-Grassmannians of signature $(1,n-1)$, whose shape is determined by the principal part of the input harmonic Maass form $f$.

In this article, we study three different evaluations of the theta lift $\Phi(f,z)$. First, by writing $f$ as a linear combination of Maass Poincar\'e series and using the unfolding argument we obtain an invariant representation of $\Phi(f,z)$ as a convergent series. Second, using the unfolding argument on (a Poincar\'e series representation of) the Siegel theta function $\Theta_{L}(\tau,z)$ we obtain the Fourier expansion of $\Phi(f,z)$. Finally, we evaluate the theta lift at \emph{special points} $w \in \Gr(L)$, which means that $w$ is defined over $\Q$. In this case, we consider the positive definite one-dimensional and negative definite $n$-dimensional sublattices 
\[
P = L \cap w, \qquad N = L \cap w^{\perp}.
\] 
Then $P \oplus N$ has finite index in $L$. For simplicity, let us assume in the introduction that $L = P \oplus N$. Then the Siegel theta function for $L$, evaluated at $w$, naturally splits as a tensor product
\[
\Theta_{L}(\tau,w) = \Theta_{P}(\tau) \otimes \Theta_{N}(\tau)
\]
of the Siegel theta functions associated to $P$ and $N$. Note that $\Theta_{P}(\tau)$ is a holomorphic unary theta series of weight $1/2$. We choose a harmonic Maass form $\widetilde{\Theta}_{P}(\tau)$ of weight $3/2$ which maps to $\Theta_{P}(\tau)$ under the $\xi$-operator and then use Stokes' Theorem to evaluate the regularized integral defining $\Phi(f,w)$. To describe the resulting expression, let us assume for the moment that $f$ is weakly holomorphic. Then we obtain
\begin{align*}
\Phi(f,w) &= \lim_{T \to \infty}\int_{\mathcal{F}_{T}}\langle f(\tau),\Theta_{P}(\tau)\otimes \Theta_{N}(\tau)\rangle v^{\frac{1-n}{2}}\frac{du dv}{v^{2}} \\
&= \CT\left(\left\langle f, \overline{\widetilde{\Theta}_{P}^{+} \otimes \Theta_{N^{-}}}\right\rangle \right),
\end{align*}
where $\CT$ denotes the constant term of a $q$-series, $\widetilde{\Theta}_{P}^{+}$ is the holomorphic part of the harmonic Maass form $\widetilde{\Theta}_{P}$, and $\Theta_{N^{-}}$ is the holomorphic theta series associated to the positive definite $n$-dimensional lattice $N^{-} = (N,-Q)$. A similar method for the evaluation of theta lifts at special points has been used by several authors, see \cite{bruinieryang, ehlenduke, schofer}, for example.

Note that $\widetilde{\Theta}_{P}^{+}(\tau)$ is a mock theta function of weight $3/2$. Hence, by comparing the three evaluations of the theta lift, one can obtain recurrences for the coefficients of mock theta functions. A classical example of such recurrences are the Hurwitz-Kronecker class number relations
\begin{align*}
\sum_{r \in \Z}H(4m-r^{2}) = 2\sigma_{1}(m) - \sum_{\substack{a,b \in \N \\ ab = m}}\min(a,b),
\end{align*}
for $m \in \N$, where $H(n)$ are the Hurwitz class numbers and $\sigma_{1}(m) = \sum_{d \mid m}d$ is the usual divisor sum. Note that the Hurwitz class numbers are the coefficients of the holomorphic part of Zagier's weight $3/2$ non-holomorphic Eisenstein series \cite{zagiereisensteinseries}. In other words, they are the coefficients of a mock theta function. We will see in Section~\ref{section applications} that the class number relations can be deduced by comparing the different evaluations of the theta lift for the even unimodular lattice of signature $(1,1)$.

The problem of finding recurrences for the Hurwitz class numbers has a long history, which goes back at least to Kronecker \cite{kronecker} and Hurwitz \cite{hurwitz2}, and has been further investigated by Eichler \cite{eichler} and Cohen \cite{cohen}, amongst others. By now, there is a very rich literature on recurrences for Hurwitz class numbers and, more generally, for the coefficients of mock theta functions. Such relations can be proved, for example, by studying Fourier coefficients of Eisenstein series (see \cite{cohen, hirzebruchzagier, williams2,williams}) or using the method of holomorphic projection (see \cite{ahlgrenandersenspt,bringmannkanehurwitz,irr,mertens1,mertens2}). Our work yields a new and quite flexible method for finding recurrences of mock theta functions, which offers a lot of freedom in the choice of the lattice $L$ and the special point $w \in \Gr(L)$.

We explain our results in an example in signature $(1,2)$ in more detail now.

\subsection{Theta lifts in signature $(1,2)$} We consider the lattice $L = \Z^{3}$ with the quadratic form $Q(a,b,c) = -b^{2}+ac$. It is an even lattice of signature $(1,2)$. The Grassmannian of $L$ can be identified with the complex upper half-plane $\H$, and the action of the modular group $\Gamma = \SL_{2}(\Z)$ by fractional linear transformations on $\H$ corresponds to its natural action on $L$ by isometries, so we can view the theta lift $\Phi(f,z)$ as a $\Gamma$-invariant function on $\H$. Moreover, vector-valued modular forms for the Weil representation associated to $L$ can be identified with scalar-valued modular forms $f$ for $\Gamma_{0}(4)$ satisfying the Kohnen plus space condition, so we can use scalar-valued modular forms $f$ as inputs for the theta lift. Finally, the set of vectors in the dual lattice $L'$ of fixed norm $-D/4$ may be identified with the set $\mathcal{Q}_{D}$ of integral binary quadratic forms $[a,b,c]$ of discriminant $D = b^{2}-4ac$. These simplifications allow us to describe our results in a more classical language in the introduction. We also refer the reader to Section~\ref{section signature 12} for more details on the translation of the results from the body of the paper to the language used in the introduction.

We let $M_{-1/2}^{!}$ be the space of weakly holomorphic modular forms $f = \sum_{n \gg -\infty}a_{f}(n)q^{n}$ of weight $-1/2$ for $\Gamma_{0}(4)$ satisfying the Kohnen plus space condition $a_{f}(n) = 0$ for $n \equiv 1,2 \pmod 4$. 
For every discriminant $D > 0$ there exists a unique weakly holomorphic modular form $f_{D} \in M_{-1/2}^{!}$ whose Fourier expansion starts $f_{D} = q^{-D}+O(1)$, and these forms constitute a basis of $M_{-1/2}^{!}$. The function $f_{D}$ can be explicitly constructed as a Maass Poincar\'e series. Using the unfolding argument, we obtain the following expression as an infinite series for the theta lift $\Phi(f_{D},z)$.

\begin{Theorem}\label{theorem unfolding introduction}
	For $z = x+iy \in \H$ we have
	\[
	\Phi(f_{D},z) = 4\sum_{[a,b,c]\in \mathcal{Q}_{D}}\left(\sqrt{D}-\frac{1}{y}|a|z|^{2}+bx+c|\arcsin\left(\frac{y\sqrt{D}}{|az^{2}+bz+c|} \right) \right).
	\]
	The series on the right-hand side converges absolutely.
	\end{Theorem}
	
	We refer the reader to Theorem~\ref{theorem unfolding} for the general result for lattices of signature $(1,n)$. We would like to remark that a direct computation shows that the image of $\Phi(f_{D},z)$ under the Maass lowering operator $L = -2iy^{2}\frac{\partial}{\partial \overline{z}}$ is a multiple of the weight $-2$ locally harmonic Maass form $\mathcal{F}_{-1,D}(z)$ investigated in \cite{bringmannkanekohnen}. 

To describe our second expression for $\Phi(f_{D},z)$ we require the $L$-function
\[
L_{D}(s) = L_{D_{0}}(s)\sum_{d\mid f}\mu(d)\chi_{D_{0}}(d)d^{-s}\sigma_{1-2s}\left(\frac{f}{d}\right),
\]
where we wrote $D = D_{0}f^{2}$ with a fundamental discriminant $D_{0}$ and a suitable $f \in \N$ (compare \cite{kz}, p.~188). Moreover, $L_{D_{0}}(s)$ denotes the usual Dirichlet $L$-function associated to the quadratic character $\chi_{D_{0}} = \big(\frac{D_{0}}{\cdot}\big)$, $\mu$ is the Moebius function and $\sigma_{s} = \sum_{d \mid n}d^{s}$ is a generalized divisor sum. The special values $L_{D}(1-k)$ at negative integers $1-k$ appear as Fourier coefficients of Cohen's Eisenstein series of weight $k+1/2$, compare \cite{cohen}. Using Cohen's Eisenstein series of weight $5/2$ and the residue theorem one can show that the constant coefficient of the weakly holomorphic modular form $f_{D}$ is given by $a_{f_{D}}(0) = -120L_{D}(-1)$. The following expression for $\Phi(f_{D},z)$ can be obtained by writing the theta function as a Poincar\'e type series and using the unfolding argument.

\begin{Theorem}\label{theorem fourier expansion introduction}
	The theta lift $\Phi(f_{D},z)$ has the Fourier expansion
	\begin{align*}
	\Phi(f_{D},z) &= -\frac{40\pi}{y}L_{D}(-1) -\frac{8\pi}{y}\sum_{\substack{[a,b,c] \in \mathcal{Q}_{D} \\ a|z|^{2}+bx+c > 0 > a}}(a|z|^{2}+bx+c) \\
	&\quad + \delta_{D = \square}\left(4\pi y D + \frac{2\pi}{y}\left(\mathbb{B}_{2}\left(\sqrt{D}x\right)+\mathbb{B}_{2}\left(-\sqrt{D}x\right) \right) \right),
	\end{align*}
	where $\mathbb{B}_{2}(x)$ is the one-periodic function which agrees with the Bernoulli polynomial $B_{2}(x) = x^{2}-x+\frac{1}{6}$ for $0 \leq x < 1$, and $\delta_{D = \square}$ equals $1$ or $0$ according to whether $D$ is a square or not. 
\end{Theorem}

The general version of this result can be found in Theorem~\ref{theorem fourier expansion}.

\begin{Remark}
	\begin{enumerate}
		\item The condition $a|z|^{2}+bx+c > 0 > a$ means that the geodesic $C_{Q} = \{z \in \H: a|z|^{2}+bx+c = 0\}$ is a semi-circle centered at the real line and that $z$ lies in the interior of the bounded component of $\H \setminus C_{Q}$. This implies that the sum on the right-hand side in the theorem is finite for fixed $z$, and vanishes for all $y > \sqrt{D}$. In particular, the above theorem gives a \emph{finite} expression for $\Phi(f_{D},z)$.
		\item It is easy to see from the above evaluation of $\Phi(f_{D},z)$ that the special value of the lift at a CM point $z$ of discriminant $d < 0$ is a \emph{rational} multiple of $\pi\sqrt{|d|}$.
	\end{enumerate}
\end{Remark}

We next describe another evaluation of $\Phi(f_{D},z)$ at CM points $z \in \H$. Note that the special points in $\Gr(L)$ precisely correspond to the CM points in $\H$. The crucial observation here is that the Siegel theta function $\Theta_{L}(\tau,z)$, evaluated at a CM point $z \in \H$, splits as a tensor product of two Siegel theta functions $\Theta_{P}(\tau)$ and $\Theta_{N}(\tau)$ associated to a positive definite one-dimensional sublattice $P$ and a negative definite two-dimensional sublattice $N$. We can then use Stokes' Theorem to evaluate the theta integral $\Phi(f_{D},z)$ in terms of the Fourier coefficients of $f_{D}(\tau)$ and the coefficients of the holomorphic part of a harmonic Maass form $\widetilde{\Theta}_{P}(\tau)$ of weight $3/2$ which maps to the unary theta function $\Theta_{P}(\tau)$ under the antilinear differential operator 
\[
\xi_{\frac{3}{2}} = 2iv^{\frac{3}{2}}\overline{\frac{\partial}{\partial \overline{\tau}}},
\]
where we wrote $\tau = u+iv \in \H$.

To simplify the exposition in the introduction, we consider the CM point $z = i$. In this case, the unary theta function $\Theta_{P}(\tau)$ is essentially equal to the usual Jacobi theta function $\theta(\tau) = \sum_{n \in \Z}q^{n^{2}}$. We let $\widetilde{\theta}(\tau)$ be a harmonic Maass form of weight $3/2$ for $\Gamma_{0}(4)$ which maps to $\theta(\tau)$ under $\xi_{3/2}$. We have the following result.

\begin{Theorem}\label{theorem splitting introduction}
	The special value of $\Phi(f_{D},z)$ at $z = i$ is given by
	\begin{align*}
	\Phi(f_{D},i) = \frac{1}{2}\sum_{\substack{d \in \Z}}a_{f_{D}}(-d)\sum_{\substack{x,y \in \Z \\ x \equiv d (2)}}a_{\widetilde{\theta}}^{+}(d-x^{2}-y^{2}),
	\end{align*}
	where $a_{f_{D}}(-d)$ and $a_{\widetilde{\theta}}^{+}(d)$ denote the coefficients of $f_{D}(\tau)$ and the holomorphic part of $\widetilde{\theta}(\tau)$, respectively.
\end{Theorem}

The general formula for lattices of signature $(1,n)$ is given in Theorem~\ref{theorem splitting computation}. Note that the holomorphic part of $\widetilde{\theta}(\tau)$ is a mock theta function of weight $3/2$. For other CM points $z \in \H$ the right-hand side will involve the coefficients of a mock theta function of weight $3/2$ of higher level (depending on the discriminant of $z$). 

In the current situation, there is a canonical preimage $\widetilde{\theta}(\tau)$ of $\theta(\tau)$. Namely, Zagier \cite{zagiereisensteinseries} showed that the generating series 
	\[
	\sum_{n = 0}^{\infty}H(n)q^{n}
	\] 
	of Hurwitz class numbers $H(n)$, with $H(0) = -1/12$, is the holomorphic part of a harmonic Maass form of weight $3/2$ for $\Gamma_{0}(4)$ which maps to $-(1/16\pi)\theta(\tau)$ under $\xi_{3/2}$. In particular, Theorem~\ref{theorem splitting introduction} yields the formula
	\[
	\Phi(f_{D},i) = -80\pi L_{D}(-1)-8\pi \sum_{\substack{x,y \in \Z \\ x \equiv D (2)}}H(D-x^{2}-y^{2}).
	\]
	Here we used that $a_{f_{D}}(0) = -120L_{D}(-1)$.
Comparing this evaluation of $\Phi(f_{D},i)$ and the one from Theorem~\ref{theorem fourier expansion introduction}, we obtain the following recurrence for the Hurwitz class numbers.

\begin{Corollary}\label{corollary hurwitz introduction}
	For every discriminant $D > 0$ we have
	\begin{align*}
	\sum_{\substack{x,y \in \Z \\ x \equiv D (2)}}H(D-x^{2}-y^{2}) = - 5L_{D}(-1)+\sum_{\substack{[a,b,c] \in \mathcal{Q}_{D} \\ a+c > 0 > a}}(a+c) -\frac{1}{12}\delta_{D = \square}(6D+1).
		\end{align*}
\end{Corollary}

To further illustrate the formulas that we can obtain using theta lifts for lattices of signature $(1,n)$, we give another example in signature $(1,1)$. We consider Andrews' smallest parts function $\spt(n)$, which counts the 
total number of smallest parts in the partitions of $n$. For $n \in \N_{0}$ we let 
\[
s(n) = \spt(n)+\frac{1}{12}(24n-1)p(n),
\]
with the partition function $p(n)$. Bringmann \cite{bringmannduke} showed that the generating series 
\[
q^{-1/24}\sum_{n=0}^{\infty}s(n)q^{n}
\]
is the holomorphic part of a harmonic Maass form of weight $3/2$ which maps to a linear combination of unary theta functions of weight $1/2$ under $\xi_{3/2}$, i.e., it is a mock modular form of weight $3/2$. By comparing our different evaluations of the theta lift for a certain lattice of signature $(1,1)$, we obtain the following recurrences for $s(n)$.

\begin{Proposition}
	We have
	\begin{align*}
	\sum_{r \in \Z}\left( \frac{12}{r}\right)s\left(m-\frac{r^{2}}{24}+\frac{1}{24} \right) = 4\sigma_{1}(m)-2\sum_{d \mid m}\left(\min(6d,m/d)-\min(3d,2m/d) \right)
	\end{align*}
	for all $m \in \N$.
\end{Proposition}

We refer the reader to Example~\ref{example spt} for more details on the proof of the above proposition. Similar recursions for the spt-function were derived by Ahlgren and Andersen in \cite{ahlgrenandersenspt}. 

Finally, we would like to remark that our different evaluations of the theta lift only yield recurrences for mock theta functions of weight $3/2$. However, one can use the methods of this paper to study a similar theta lift obtained by replacing $\Theta_{L}(\tau,z)$ with a modified theta function $\Theta_{L}^{*}(\tau,z)$, in order to find recurrences for mock theta functions of weight $1/2$. We will sketch this method, which yields recurrences for some of Ramanujan's classical mock theta functions, in Section~\ref{section mock theta functions}.

This work is organized as follows. In Section~\ref{section preliminaries} we collect the necessary preliminaries about vector-valued harmonic Maass forms for the Weil representation, Siegel theta functions and their splittings at special points, and different models of the Grassmannian $\Gr(L)$. In Section~\ref{section theta lifts}, which is the technical heart of this work, we compute the theta lift for lattices of signature $(1,n)$ in three different ways. Finally, in Section~\ref{section applications} we discuss several applications of these different evaluations and obtain recurrences for the coefficients of some mock theta functions.

\section{Preliminaries}\label{section preliminaries}

Throughout this section, we let $L$ be an even lattice of signature $(p,q)$ with quadratic form $Q$ and associated bilinear form $(\cdot,\cdot)$. The dual lattice of $L$ will be denoted by $L'$.

\subsection{The Weil representation} Let $\C[L'/L]$ be the group ring of $L$, which is generated by the standard basis vectors $\e_{\gamma}$ for $\gamma \in L'/L$. We let $\langle \cdot ,\cdot \rangle$ be the natural inner product on $\C[L'/L]$ which is antilinear in the second variable. Let $\Mp_{2}(\Z)$ be the metaplectic double cover of $\SL_{2}(\Z)$, realized as the set of pairs $(M,\phi)$ with $M = \left(\begin{smallmatrix}a & b \\ c & d \end{smallmatrix} \right) \in \SL_{2}(\Z)$ and $\phi:\H\to\C$ a holomorphic function with $\phi^{2}(\tau) = c\tau + d$. The Weil representation $\rho_{L}$ of $\Mp_{2}(\Z)$ associated to $L$ is defined on the generators $T = \left(\left(\begin{smallmatrix}1 & 1 \\ 0 & 1 \end{smallmatrix} \right), 1 \right)$ and $S = \left(\left(\begin{smallmatrix}0 & -1 \\ 1 & 0 \end{smallmatrix} \right), \sqrt{\tau} \right)$ by 
\begin{align*}
\rho_{L}(T)\e_{\gamma} = e(Q(\gamma))\e_{\gamma},  \qquad  \rho_{L}(S)\e_{\gamma}= \frac{\sqrt{i}^{q-p}}{\sqrt{|L'/L|}}\sum_{\beta \in L'/L}e(-(\beta,\gamma))\e_{\beta},
\end{align*}
where we put $e(x) = e^{2\pi i x}$ for $x \in \C$. The Weil representation $\rho_{L^{-}}$ associated to the lattice $L^{-} = (L,-Q)$ will be called the dual Weil representation associated to $L$. 

\subsection{Operators on vector-valued modular forms}\label{section operators}

For $k \in \frac{1}{2}\Z$ we let $A_{k,L}$ be the set of all functions $f: \H \to \C$ that transform like modular forms of weight $k$ for $\rho_{L}$, which means that $f$ is invariant under the slash operator 
\[
f|_{k,L}(M,\phi)= \phi(\tau)^{-2k}\rho_{L}(M,\phi)^{-1}f(M\tau)
\]
for $(M,\phi) \in \Mp_{2}(\Z)$. Let $K \subset L$ be a sublattice of finite index. Then we have the inclusions $K \subset L \subset L' \subset K'$ and thus $L/K \subset L'/K \subset K'/K$. We have the natural map $L'/K \to L'/L, \gamma \mapsto \overline{\gamma}$. There are maps
\begin{align*}
\res_{L/K}&: A_{k,L} \to A_{k,K}, \quad f\mapsto f_{K}, \\
\tr_{L/K}&: A_{k,K} \to A_{k,L}, \quad g\mapsto g^{L},
\end{align*}
which are defined for $f \in A_{k,L}$ and $\gamma \in K'/K$ by
\begin{align*}
(f_{K})_{\gamma} = \begin{cases} f_{\overline{\gamma}}, & \text{if } \gamma \in L'/K, \\ 0, & \text{if } \gamma \notin L'/K,\end{cases}
\end{align*}
and for $g \in A_{k,K}$ and $\overline{\gamma} \in L'/L$ by
\begin{align*}
(g^{L})_{\overline{\gamma}} = \sum_{\beta \in L/K}g_{\beta+\gamma}.
\end{align*}
They are adjoint with respect to the inner products on $\C[L'/L]$ and $\C[K'/K]$. We refer the reader to \cite{bruinierconverse}, Section~3, or \cite{scheithauer}, Section~4, for more details.

\subsection{Harmonic Maass forms}\label{section harmonic maass forms}

Recall from \cite{bruinierfunke04} that a harmonic Maass form of weight $k \in \frac{1}{2}\Z$ for $\rho_{L}$ is a smooth function $f: \H \to \C$ which is annihilated by the weight $k$ Laplace operator 
\[
\Delta_{k} = -v^{2}\left(\frac{\partial^{2}}{\partial u^{2}}+\frac{\partial^{2}}{\partial v^{2}} \right)+ikv\left(\frac{\partial}{\partial u}+i\frac{\partial}{\partial v} \right), \qquad (\tau = u + iv \in \H),
\]
which transforms like a modular form of weight $k$ for $\rho_{L}$, and which is at most of linear exponential growth at the cusp $\infty$. The space of harmonic Maass forms of weight $k$ for $\rho_{L}$ is denoted by $H_{k,L}$. We let $M_{k,L}^{!}$ be subspace of weakly holomorphic modular forms, which consists of the forms that are holomorphic on $\H$. The antilinear differential operator
\[
\xi_{k} = 2iv^{k}\overline{\frac{\partial}{\partial \overline{\tau}}}
\]
maps $H_{k,L}$ onto $M_{2-k,L^{-}}^{!}$. We let $H_{k,L}^{\hol}$ and $H_{k,L}^{\cusp}$ be the subspaces of $H_{k,L}$ which are mapped to the space $M_{2-k,L^{-}}$ of holomorphic modular forms or the space $S_{2-k,L^{-}}$ of cusp forms under $\xi_{k}$, respectively. For $k \leq 0$ every $f \in H_{k,L}^{\hol}$ decomposes as a sum $f = f^{+} + f^{-}$ of a holomorphic and a non-holomorphic part, having Fourier expansions of the form
\begin{align}\label{eq Fourier expansion}
\begin{split}
f^{+}(\tau) &= \sum_{\gamma \in L'/L}\sum_{\substack{n \in \Q \\ n \gg -\infty}}a_{f}^{+}(\gamma,n)q^{n}\e_{\gamma}, \\
f^{-}(\tau) &=  \sum_{\gamma \in L'/L}\bigg(a_{f}^{-}(\gamma,0)v^{1-k}+\sum_{\substack{n \in \Q \\ n < 0}}a_{f}^{-}(\gamma,n)\Gamma(1-k,4\pi|n|v)q^{n}\bigg)\e_{\gamma},
\end{split}
\end{align}
where $a_{f}^{\pm}(\gamma,n) \in \C$, $q = e^{2\pi i \tau}$, and $\Gamma(s,x) = \int_{x}^{\infty}e^{-t}t^{s-1}dt$ is the incomplete Gamma function. Note that $f \in H_{k,L}^{\cusp}$ is equivalent to $a_{f}^{-}(\gamma,0) = 0$ for all $\gamma \in L'/L$.

\subsection{Maass Poincar\'e series}\label{section poincare series} Examples of harmonic Maass forms of weight $k \leq 0$ can be constructed using Maass Poincar\'e series, compare \cite{brhabil}, Section~1.3. For $s \in \C$ and $v > 0$ we let
\[
\mathcal{M}_{s}(v) = v^{-k/2}M_{-k/2,s-1/2}(v),
\]
with the usual $M$-Whittaker function. For $\beta \in L'/L$ and $m \in \Z - Q(\beta)$ with $m > 0$, and $s \in \C$ with $\Re(s) > 1$ we define the Maass Poincar\'e series
\[
F_{\beta,m}(\tau,s) = \frac{1}{2\Gamma(2s)}\sum_{(M,\phi) \in \widetilde{\Gamma}_{\infty}\backslash \Mp_{2}(\Z)}\mathcal{M}_{s}(4\pi m v)e(-mu)\e_{\beta}|_{k,L}(M,\phi),
\]
where $\widetilde{\Gamma}_{\infty}$ is the subgroup of $\Mp_{2}(\Z)$ generated by $T=\left(\left(\begin{smallmatrix}1 & 1 \\ 0 & 1 \end{smallmatrix} \right), 1 \right)$. It converges absolutely for $\Re(s) > 1$, it transforms like a modular form of weight $k$ for $\rho_{L}$, and it is an eigenform of the Laplace operator $\Delta_{k}$ with eigenvalue $s(1-s)+(k^{2}-2k)/4$. For $k = 0$ the Maass Poincar\'e series is also called a Niebur Poincar\'e series \cite{niebur}. It has a meromorphic continuation in $s$ to $\C$ which is obtained from its Fourier expansion and which is holomorphic at $s = 1$. For all $k \leq 0$ the special value
\[
F_{\beta,m}(\tau) = F_{\beta,m}\left(\tau,1-\frac{k}{2}\right)
\]
defines a harmonic Maass form in $H_{k,L}^{\cusp}$ whose Fourier expansion starts with
\[
F_{\beta,m}(\tau) = q^{-m}(\e_{\beta}+\e_{-\beta}) + O(1).
\]
In particular, for $k\leq 0$ every harmonic Maass form $f \in H_{k,L}^{\cusp}$ with Fourier expansion as in \eqref{eq Fourier expansion} can be written as a linear combination
\begin{align}\label{eq linear combination Maass Poincare}
f(\tau) = \frac{1}{2}\sum_{\beta \in L'/L}\sum_{m > 0}a_{f}^{+}(\beta,-m)F_{\beta,m}(\tau) + \mathfrak{c}
\end{align}
of Maass Poincar\'e series and a constant $\mathfrak{c} \in \C[L'/L]$, which may be non-zero only if $k = 0$.

\subsection{Siegel theta functions and special points}\label{section theta functions}

As before, we let $L$ be an even lattice of signature $(p,q)$. Let $\Gr(L)$ be the Grassmannian of positive definite $p$-dimensional subspaces of $V = L \otimes \R$. 
The Siegel theta function associated to $L$ is defined by
\begin{align*}
\Theta_{L}(\tau,z) = v^{q/2}\sum_{\lambda \in L'}e(Q(\lambda_{z})\tau+Q(\lambda_{z^{\perp}})\overline{\tau})\e_{\lambda + L},
\end{align*}
where $\tau = u + iv \in \H$ and $z \in \Gr(L)$, and $\lambda_{z}$ denotes the orthogonal projection of $\lambda$ to $z$. 
The Siegel theta function transforms like a modular form of weight $\frac{p-q}{2}$ for $\rho_{L}$ in $\tau$ (see \cite{borcherds}, Theorem~4.1) and is invariant in $z$ under the subgroup of $\O(L)$ fixing $L'/L$.

\begin{Definition}
We call $w\in \Gr(L)$ a \emph{special point} if it is defined over $\Q$, that is, $w \in L \otimes \Q$.
\end{Definition}

For a special point $w \in \Gr(L)$ its orthogonal complement $w^{\perp}$ in $V$ is also defined over $\Q$ and we obtain the rational splitting $L \otimes \Q = w \oplus w^{\perp}$ which yields the positive and negative definite lattices
\[
P = L \cap w, \qquad N = L \cap w^{\perp}.
\]
%
Note that $P \oplus N$ is a sublattice of $L$ of finite index. The Siegel theta functions associated to $L$ and $P \oplus N$ are related by
\begin{align}\label{eq theta relation}
\Theta_{L} = \left(\Theta_{P\oplus N}\right)^{L},
\end{align}
with the trace operator defined in Section~\ref{section operators}. Moreover, at the special point $w$ the Siegel theta function associated to $P \oplus N$ splits as a tensor product
\begin{align}\label{splitting siegel theta}
\Theta_{P \oplus N}(\tau,w) = \Theta_{P}(\tau) \otimes \Theta_{N}(\tau),
\end{align}
where we identified $\C[(P\oplus N)'/(P\oplus N)]$ with $\C[P'/P] \otimes \C[N'/N]$. These two observations can be proved by direct calculations using the definition of the Siegel theta function.

\subsection{Unary theta series and Atkin-Lehner operators}
\label{section unary theta series}

For $d \in \N$ the one-dimensional positive definite lattice $\Z(d) = (\Z,dx^{2})$ has level $4d$ and its discriminant group is isomorphic to $\Z/2d\Z$ with the $\Q/\Z$-valued quadratic form $x \mapsto x^{2}/4d$. The unary theta series
\[
\theta_{1/2,d}(\tau) = \sum_{r(2d)}\sum_{\substack{n \in \Z \\ n \equiv r (2d)}}q^{n^{2}/4d}\e_{r}, \qquad \theta_{3/2,d}(\tau) = \sum_{r(2d)}\sum_{\substack{n \in \Z \\ n \equiv r (2d)}}n q^{n^{2}/4d}\e_{r},
\]
are holomorphic modular forms for the Weil representation associated to $\Z(d)$ of weight $1/2$ and $3/2$, respectively.

For an exact divisor $c \mid \mid d$ the Atkin-Lehner involution $\sigma_{c}$ on $\Z/2d\Z$ is the map defined by the equations $\sigma_{c}(x) \equiv -x (2c)$ and $\sigma_{c}(x) \equiv x (2d/c)$. The Atkin-Lehner involutions act on vector-valued modular forms for the Weil representation associated to $\Z(d)$ by 
\[
\bigg(\sum_{r (2d)}f_{r}(\tau)\e_{r} \bigg)^{\sigma_{c}} = \sum_{r (2d)}f_{\sigma_{c}(r)}(\tau)\e_{r}.
\]

\subsection{The upper half-space model of $\Gr(L)$} 
\label{section half-space}

Let $L$ be a lattice of signature $(1,n)$ with $n \geq 1$. In this case, the Grassmannian $\Gr(L)$ can be identified with hyperbolic $n$-space, as we now explain. For simplicity, we assume that $L$ contains two primitive isotropic vectors $\ell$ and $\ell'$ with $(\ell,\ell') = 1$, that is, $L$ splits a hyperbolic plane over $\Z$. We consider the negative definite lattice
\[
K = L \cap \ell^{\perp} \cap \ell'^{\perp}.
\]
Note that $L'/L \cong K'/K$. We can identify $\Gr(L)$ with one of the two connected components of the set of norm $1$ vectors in $V = L \otimes \R$. We choose the component
\[
V_{1} = \{v_{1} \in V: |v_{1}| = 1, (v_{1},\ell) > 0\}
\]
and identify $V_{1}$ with the Grassmannian by mapping $z \in \Gr(L)$ to 
\[
v_{1} = \frac{\ell_{z}}{|\ell_{z}|} \in V_{1}.
\]
This is called the \emph{hyperboloid model} of hyperbolic space. Furthermore, we can identify $V_{1}$ with the \emph{upper half-space model}
\[
\H^{n}= (K\otimes \R) \times \R_{>0}= \{(x,y): x \in K \otimes \R, y > 0\}
\]
by mapping $(x,y) \in \H^{n}$ to the norm $1$ vector
\[
v_{1} = \frac{1}{\sqrt{2}y}x + \frac{1}{\sqrt{2}y}\ell' + \left( \frac{y}{\sqrt{2}}-\frac{Q(x)}{\sqrt{2}y}\right)\ell \in V_{1}.
\]
Note that we can write $\lambda \in V$ as $\lambda = b+a\ell' + c\ell$ with $ b\in K \otimes \R$ and $a,c \in \R$. Then $Q(\lambda) = Q(b) + ac$. In the hyperboloid model and the upper half-space model the quantity $Q(\lambda_{z})$ is given by
\begin{align}\label{eq Qlambdaz}
Q(\lambda_{z}) = \frac{1}{2}(\lambda,v_{1})^{2} = \frac{1}{4y^{2}}(a(y^{2}-Q(x)))+(b,x)+c)^{2}.
\end{align}
This can be used to translate the formulas from Theorem~\ref{theorem unfolding} and Theorem~\ref{theorem fourier expansion} into the corresponding statements in the introduction.

\section{The regularized theta lift}\label{section theta lifts}

Let $L$ be a Lorentzian lattice, that is, an even lattice of signature $(1,n)$ with $n\geq 1$. Throughout, we let $k = (1-n)/2$ and we let $f \in H_{k,L}^{\cusp}$ be a harmonic Maass form with Fourier expansion as in \eqref{eq Fourier expansion}. Following Borcherds \cite{borcherds}, for $z \in \Gr(L)$ we define the regularized theta lift
\begin{align}\label{eq definition theta lift}
\Phi(f,z) = \lim_{T \to \infty}\int_{\mathcal{F}_{T}}\langle f(\tau),\Theta_{L}(\tau,z)\rangle v^{k}\frac{du dv}{v^{2}},
\end{align}
where $\mathcal{F}_{T}$ denotes the standard fundamental domain for $\SL_{2}(\Z)$, cut off at height $v = T$. By the results of \cite{borcherds}, Section~10, and \cite{brhabil}, Chapter~3, the theta lift is real-analytic up to singularities along the union of the sets
\[
H(\beta,m) = \bigcup_{\substack{\lambda \in L + \beta \\ Q(\lambda) = -m}}\lambda^{\perp} 
\]
for those $\beta \in L'/L$ and $m > 0$ with $a_{f}^{+}(\beta,-m) \neq 0$. Here $\lambda^{\perp}$ denotes the set of all $z \in \Gr(L)$ perpendicular to $\lambda$. However, $\Phi(f,z)$ is continuous on all of $\Gr(L)$.


In this section we evaluate the regularized theta lift of $f \in H_{k,L}^{\cusp}$ in three ways. For simplicity, we assume throughout that the constant $\mathfrak{c}$ in \eqref{eq linear combination Maass Poincare} vanishes also in the case $n = 1$.
\begin{enumerate}
	\item We write $f$ as a linear combination of Maass Poincar\'e series $F_{\beta,m}$ as in \eqref{eq linear combination Maass Poincare} and use the unfolding argument to compute $\Phi(F_{\beta,m},z)$. This yields an invariant representation of $\Phi(f,z)$ as an infinite series. 
	\item By unfolding against the theta function and using the shape of the singularities of the theta lift, one can compute the Fourier expansion of $\Phi(f,z)$ on all of $\Gr(L)$.
	\item At special points $w \in \Gr(L)$ we can compute $\Phi(f,w)$ by using the splitting \eqref{splitting siegel theta} of the Siegel theta function and applying Stokes' theorem. We obtain a formula for $\Phi(f,w)$ involving the coefficients of a mock theta function of weight $3/2$. 
\end{enumerate}


\subsection{An invariant representation of the theta lift}

Let $k = (1-n)/2$ as before. We have the following representation of $\Phi(f,z)$ as an infinite series.

\begin{Theorem}\label{theorem unfolding} For $f \in H_{k,L}^{\cusp}$ and $z \in \Gr(L)$ we have
\begin{align*}
		\Phi(f,z) &= \frac{2\sqrt{\pi}\Gamma\left(\frac{n}{2}\right)}{\Gamma\left(\frac{n+3}{2}\right)}\sum_{\substack{\lambda \in L' \\ Q(\lambda) < 0}}a_{f}^{+}(\lambda ,Q(\lambda))\frac{|Q(\lambda)|^{\frac{n+1}{2}}}{ |Q(\lambda_{z^{\perp}})|^{\frac{n}{2}} } \ _{2}F_{1}\left(\frac{n}{2},1,\frac{n+3}{2};\frac{Q(\lambda)}{Q(\lambda_{z^{\perp}})}\right),
	\end{align*}
	where $_{2}F_{1}(a,b,c;x)$ denotes Gauss' hypergeometric function. The series on the right-hand side converges absolutely.
\end{Theorem}


\begin{Remark}\label{remark unfolding}
	The hypergeometric function appearing in Theorem~\ref{theorem unfolding} can be evaluated in terms of more elementary functions for small values of $n$. For example, for $n = 1$ and $f \in H_{0,L}^{\cusp}$ we find
	\begin{align}\label{eq unfolding 1}
	\Phi(f,z) = 4\pi \sum_{\substack{\lambda \in L' \\ Q(\lambda) < 0}}a_{f}^{+}(\lambda,Q(\lambda))\left(\sqrt{|Q(\lambda_{z^{\perp}})|}-\sqrt{|Q(\lambda_{z})|} \right),
	\end{align}
	and for $n = 2$ and $f \in H_{-1/2,L}^{\cusp}$ we obtain
	\begin{align}\label{eq unfolding 2} 
	\Phi(f,z) = 8\sum_{\substack{\lambda \in L' \\ Q(\lambda) < 0}}a_{f}^{+}(\lambda,Q(\lambda))\left(\sqrt{|Q(\lambda)|}-\sqrt{Q(\lambda_{z})}\arcsin\left( \sqrt{\frac{Q(\lambda)}{Q(\lambda_{z^{\perp}})}}\right) \right).
	\end{align}
	Using $Q(\lambda_{z^{\perp}}) = Q(\lambda)-Q(\lambda_{z})$ together with \eqref{eq Qlambdaz} one can translate the above series representation of $\Phi(f,z)$ into the upper half-space model, which yields Theorem~\ref{theorem unfolding introduction}.
\end{Remark}

\begin{proof}[Proof of Theorem~\ref{theorem unfolding}]
	Write $f$ as a linear combination of Maass Poincar\'e series $F_{\beta,m}$ as in \eqref{eq linear combination Maass Poincare}. Recall that we assume $\mathfrak{c} = 0$ also in the case $n = 1$. We have
	\[ 
	\Phi(F_{\beta,m},z) = \big[\Phi(F_{\beta,m}(\cdot,s),z)\big]_{s = 1-\frac{k}{2}}.
	\]
	As in \cite{brhabil}, Theorem 2.13, we can compute $\Phi(F_{\beta,m}(\cdot,s),z)$ by unfolding against $F_{\beta,m}(\tau,s)$ and get
	\begin{align*}
	\Phi(F_{\beta,m}(\cdot,s),z) &= 2(4\pi m)^{s-\frac{k}{2}}\frac{\Gamma\left(s+\frac{n-3}{4}\right)}{\Gamma(2s)}\\
	&\quad \times\sum_{\substack{\lambda \in L + \beta \\ Q(\lambda) = -m}}(4\pi |Q(\lambda_{z^{\perp}})|)^{\frac{3-n}{4}-s}\ _{2}F_{1}\left(s+\frac{n-3}{4},s-\frac{n-1}{4},2s;\frac{m}{|Q(\lambda_{z^{\perp}})|}\right).
	\end{align*}
	We plug in $s = 1-k/2 = (3+n)/4$ to obtain the stated formula.
	
	For $k < 0$ the absolute convergence of the series on the right-hand side of the theorem is contained in \cite{brhabil}, Theorem~2.13, and for $k = 0$, that is, signature $(1,1)$, it follows from the more explicit formulas for the series given in Theorem~\ref{theorem main identity isotropic} and Theorem~\ref{theorem main identity anisotropic} (together with Remark~\ref{remark anisotropic}) below.
\end{proof}

\subsection{The Fourier expansion of the theta lift} In order to state the Fourier expansion of the theta lift we need to assume that $L$ contains a primitive isotropic vector $\ell$. Moreover, to simplify the exposition, we further assume in this section that $L$ contains another primitive isotropic vector $\ell'$ with $(\ell,\ell') = 1$ and a negative definite sublattice $K$ of rank $n-1$ such that
\[
L = (\Z \ell \oplus \Z \ell') \oplus^{\perp} K
\] 
as an orthogonal direct sum. In other words, we assume that $L$ splits a hyperbolic plane over $\Z$. Note that in this case we have $L'/L \cong K'/K$. 

It is convenient now to use the upper half-space model of $\Gr(L)$, compare Section~\ref{section half-space}. The theta lift $\Phi(f,(x,y))$, viewed as a function on the upper half-space $\H^{n}$, has the following Fourier expansion. 

\begin{Theorem}\label{theorem fourier expansion}
	For $f \in H_{k,L}^{\cusp}$ and $(x,y) \in \H^{n}$ we have
	\begin{align*}
	\Phi(f,(x,y)) &= \frac{8\pi y}{(n-1)}\sum_{b \in K'}a_{f}^{+}(b,Q(b))|Q(b)|+\frac{2\pi}{y}\sum_{b \in K'}a_{f}^{+}(b,Q(b))\mathbb{B}_{2}\left((b,x)\right) \\
	&\quad +2\sqrt{2}\left(\sqrt{2}\pi y\right)^{-k}\sum_{b \in K'\setminus \{0\}}a_{f}^{-}(b,Q(b))|b|^{1-k}\sum_{n =1}^{\infty}\frac{e(n(b,x))}{n^{k+1}}K_{1-k}\left(2\sqrt{2}\pi y |b|\right)\\
	&\quad -\frac{8\pi}{y}\!\!\!\!\!\!\!\!\!\! \sum_{\substack{a,c \in \Z, b \in K' \\ Q(b)+ac < 0\\ a(y^{2}-Q(x))+(b,x)+c > 0 > a}}\!\!\!\!\!\!\!\!\!\! a_{f}^{+}(b,Q(b)+ac)\left(a(y^{2}-Q(x))+(b,x)+c \right),
	\end{align*}
	where $\mathbb{B}_{2}(x)$ is the one-periodic function which agrees with the Bernoulli polynomial $B_{2}(x) = x^{2}-x+1/6$ for $0 \leq x < 1$ and $K_{1-k}(x)$ denotes the $K$-Bessel function.
\end{Theorem}

\begin{Remark} 
	\begin{enumerate}
		\item If $f \in M_{k,L}^{!}$ is weakly holomorphic, then the second line in the theorem vanishes. All remaining sums are finite, so the above expansion gives a finite expression for $\Phi(f,(x,y))$ for weakly holomorphic $f$.
		\item If we write $\lambda \in V$ with $Q(\lambda) < 0$ as $\lambda = b + a\ell' + c \ell$ with $b \in K \otimes \R$ and $a,c \in \R$, then the set $\lambda^{\perp}$ (consisting of all $z \in \Gr(L)$ orthogonal to $\lambda$) is in the upper half-space model given by
		\[
		\lambda^{\perp} = \{(x,y) \in \H^{n}: a(y^{2}-Q(x))+(b,x)+c = 0\}.
		\]
		The third line in the above Fourier expansion reproduces the singularities of $\Phi(f,(x,y))$ along those $\lambda^{\perp}$ with $a \neq 0$, while the sum involving the function $\mathbb{B}_{2}(x)$ in the first line gives the singularities along those $\lambda^{\perp}$ with $a = 0$.
	\end{enumerate}
\end{Remark}

\begin{proof}[Proof of Theorem~\ref{theorem fourier expansion}]
	We first show that the sum in the last line of the theorem is finite for fixed $(x,y) \in \H^{n}$, and vanishes for all $y$ big enough, independently of $x$. We can split the sum into a finite sum over those $m < 0$ with $a_{f}^{+}(b,m) \neq 0$ for some $b \in K'$, and a sum over $a,c \in \Z, b \in K'$ with $Q(b)+ac = m$ and $a(y^{2}-Q(x))+(b,x)+c > 0 > a$. For $a,c \in \Z$ and $b\in K'$ with $a < 0$ and $Q(b)+ac = m$, we have 
	\[
	-aQ(x)+(b,x)+c \leq -aQ(x)+2\sqrt{|Q(b)|}\sqrt{|Q(x)|}+c \leq  \frac{m}{a}
	\] 
	for all $x \in K \otimes \R$, where we used the Cauchy-Schwarz inequality in the first step. In particular, for fixed $(x,y) \in \H^{n}$ we have
	\[
	a(y^{2}-Q(x))+(b,x)+c \leq ay^{2}+\frac{m}{a},
	\]
	which is positive only for finitely many $a \in \Z$, $a< 0$. Furthermore, this shows that $a(y^{2}-Q(x))+(b,x)+c > 0 > a$ implies $y \leq \sqrt{|m|}$, so each summand in the third line of the theorem vanishes for $y > \sqrt{|m|}$. Fix one of the finitely many $a < 0$ for which $ay^{2}+\frac{m}{a}$ is positive. Then we can write
	\[
	a(y^{2}-Q(x))+(b,x)+c = a(y^{2}-Q(x))+(b,x)+\frac{m-Q(b)}{a}.
	\]
	Since $(b,x)-\frac{Q(b)}{a}$ is positive only for finitely many $b \in K'$, we see that the sum in the third line of the theorem is finite for fixed $(x,y) \in \H^{n}$.
	

	We now write $f$ as a linear combination of Maass Poincar\'e series $F_{\beta,m}$ in order to apply results from \cite{brhabil}. First, equation (3.11) in \cite{brhabil} gives an expression for the Fourier expansion of $\Phi(F_{\beta,m},v_{1})$ for $v_{1} \in V_{1}$ (the hyperboloid model of $\Gr(L)$, see Section~\ref{section half-space}) with $(\ell,v_{1})$ small enough. The constant $\Phi_{\beta,m}^{K}$ appearing in this formula can be evaluated using \cite{brhabil}, Theorem~2.13. We obtain for $(\ell,v_{1})$ small enough the expansion
	\begin{align*}
	\Phi(f,v_{1}) &= \frac{8\pi}{\sqrt{2}(n-1)(\ell,v_{1})} \sum_{\lambda \in K'}a_{f}^{+}(\lambda,Q(\lambda))|Q(\lambda)|\\
	&\quad + 2\sqrt{2}\pi(\ell,v_{1})\sum_{\lambda \in K'}a_{f}^{+}(\lambda,Q(\lambda))\mathbb{B}_{2}\left(\frac{(\lambda,v_{1})}{(\ell,v_{1})}\right) \\
	&\quad + 2\sqrt{2}\left(\frac{\pi}{(\ell,v_{1})} \right)^{-k}\sum_{\lambda \in K'\setminus \{0\}}a_{f}^{-}(\lambda,Q(\lambda))|\lambda|^{1-k}\sum_{n=1}^{\infty}\frac{e\left(n\frac{(\lambda,v_{1})}{(\ell,v_{1})} \right)}{n^{k+1}}K_{1-k}\left(\frac{2\pi n |\lambda|}{(\ell,v_{1})} \right).
	\end{align*}
	Recall that we can write $\lambda \in V$ as $\lambda = b + a\ell' + c \ell$ with $b \in K\otimes \R$ and $a,c \in \R$. In the upper half-space model we have
	\begin{align*}
	(\ell,v_{1}) = \frac{1}{\sqrt{2}y}, \quad (\lambda,v_{1}) = \frac{1}{\sqrt{2}y}\left(a(y^{2}-Q(x))+(b,x)+c\right).
	\end{align*}
	Hence, in the coordinates of $\H^{n}$, we find for $y$ big enough the expansion
	\begin{align*}
	\Phi(f,(x,y)) &= \frac{8\pi y}{(n-1)}\sum_{b \in K'}a_{f}^{+}(b,Q(b))|Q(b)|+\frac{2\pi}{y}\sum_{b \in K'}a_{f}^{+}(b,Q(b))\mathbb{B}_{2}\left((b,x)\right) \\
	&+2\sqrt{2}\left(\sqrt{2}\pi y \right)^{-k}\sum_{b \in K'\setminus\{0\}}a_{f}^{-}(b,Q(b))|b|^{1-k}\sum_{n=1}^{\infty}\frac{e\left(n(b,x) \right)}{n^{k+1}}K_{1-k}\left(2\sqrt{2}\pi n y |b| \right).
	\end{align*}
	On the other hand, by \cite{brhabil}, Theorem~2.11, the theta lift $\Phi(f,(x,y))$ has a singularity of type
	\begin{align}\label{eq singularity}
	-\frac{2\pi}{y}\!\!\!\!\!\!\!\!\!\!\sum_{\substack{a,c \in \Z, b \in K' \\ Q(b)+ac < 0 \\ a(y_{0}^{2}-Q(x_{0}))+(b,x_{0})+c = 0}}\!\!\!\!\!\!\!\!\!\!a_{f}^{+}(b,Q(b)+ac)\left(a(y^{2}-Q(x))+(b,x)+c \right)
	\end{align}
	at a point $(x_{0},y_{0}) \in \H^{n}$. Recall that this means that there exists a neighbourhood $U$ of $(x_{0},y_{0})$ such that $\Phi(f,(x,y))$ and the expression in \eqref{eq singularity} are defined on a dense subset of $U$ and such that their difference extends to a real-analytic function on all of $U$. It is not hard to see that the expression on the right-hand side in the theorem has the same singularities as $\Phi(f,(x,y))$ on $\H^{n}$ and agrees with $\Phi(f,(x,y))$ for $y$ big enough. In other words, the difference between $\Phi(f,(x,y))$ and the expression on the right-hand side of the theorem extends to a real-analytic function on all of $\H^{n}$ and vanishes for $y$ big enough, and hence has to vanish identically on all of $\H^{n}$. This finishes the proof.
\end{proof}

 \subsection{Evaluation of the theta lift at special points} 
  Let $w \in \Gr(L)$ be a special point and let $P = L \cap w$ and $N = L \cap w^{\perp}$ be the positive and negative definite sublattices corresponding to $w$ as in Section~\ref{section theta functions}. Recall that $k = (1-n)/2$.
  
\begin{Theorem}\label{theorem splitting computation}
	For $f \in H_{k,L}^{\cusp}$ and a special point $w \in \Gr(L)$ we have the evaluation
	\begin{align*}
	\Phi(f,w) &= \CT\left(\left\langle f_{P \oplus N}^{+}, \overline{\widetilde{\Theta}_{P}^{+}\otimes \Theta_{N^{-}}}\right\rangle \right) \\
	&\quad-\lim_{T\to \infty}\int_{\mathcal{F}_{T}}\left\langle  \widetilde{\Theta}_{P}(\tau) \otimes \Theta_{N}(\tau), \xi_{2-k}f_{P \oplus N}(\tau)\right\rangle v^{2-k}\frac{du dv}{v^{2}},
	\end{align*}
	where $\CT$ denotes the constant coefficient of a Fourier series, $f_{P \oplus N}$ denotes the image of $f$ under the restriction map defined in Section~\ref{section operators}, and $\widetilde{\Theta}_{P} \in H_{3/2,P^{-}}^{\hol}$ is a $\xi_{3/2}$-preimage of $\Theta_{P}$.
\end{Theorem}

\begin{Remark}\label{remark constant term}
	By plugging in the Fourier expansions of $f$ and $\widetilde{\Theta}_{P}^{+}$, the constant term on the right-hand side of Theorem~\ref{theorem splitting computation} can be evaluated more explicitly as
	\begin{align*}
	\CT\left(\left\langle f_{P \oplus N}^{+}, \overline{\widetilde{\Theta}_{P}^{+}\otimes \Theta_{N^{-}}}\right\rangle \right) &= \sum_{\substack{\alpha \in P'/P \\ \beta \in N'/N \\ \alpha+\beta \in L'}}\sum_{n \in \Z-Q(\alpha+\beta)}a_{f}^{+}(\alpha + \beta,-n)\sum_{\lambda \in N+\beta}a_{\widetilde{\Theta}_{P}}^{+}(\alpha,n+Q(\lambda)).
	\end{align*}
	Note that all sums on the right-hand side are finite. Theorem~\ref{theorem splitting introduction} in the introduction can easily be derived from this expression.
\end{Remark}

\begin{proof}[Proof of Theorem~\ref{theorem splitting computation}]
	The proof is similar to the proof of \cite{ehlenduke}, Theorem 3.5, so we omit some details for brevity. We first compute for arbitrary $z \in \Gr(L)$
	\begin{align*}
	\Phi(f,z) &= \lim_{T\to \infty}\int_{\mathcal{F}_{T}}\langle f(\tau),\Theta_{L}(\tau,z)\rangle v^{k} \frac{du dv}{v^{2}} \\
	&= \lim_{T\to \infty}\int_{\mathcal{F}_{T}}\langle f(\tau),(\Theta_{P\oplus N})^{L}(\tau,z)\rangle v^{k}\frac{du dv}{v^{2}} \\
	&= \lim_{T\to \infty}\int_{\mathcal{F}_{T}}\langle f_{P\oplus N}(\tau),\Theta_{P\oplus N}(\tau,z)\rangle v^{k}\frac{du dv}{v^{2}},
	\end{align*}
	where we used the relation \eqref{eq theta relation} and the fact that the trace and restriction operators from Section~\ref{section operators} are adjoint to each other. Now we plug in $z = w$ and use the splitting \eqref{splitting siegel theta} to obtain
	\begin{align*}
	\Phi(f,w) = \lim_{T\to \infty}\int_{\mathcal{F}_{T}}\langle f_{P\oplus N}(\tau),\Theta_{P}(\tau)\otimes\Theta_{N}(\tau)\rangle v^{k}\frac{du dv}{v^{2}}.
	\end{align*}
	Let $\widetilde{\Theta}_{P} \in H_{3/2,P^{-}}^{\hol}$ be a harmonic Maass form of weight $3/2$ for $\rho_{P^{-}}$ such that $\xi_{3/2}\widetilde{\Theta}_{P} = \Theta_{P}$. Then we can write
	\begin{align*}
	\Phi(f,w) &= \lim_{T\to \infty}\int_{\mathcal{F}_{T}}\left\langle f_{P\oplus N}(\tau),(\xi_{3/2}\widetilde{\Theta}_{P}(\tau))\otimes\overline{\Theta_{N^{-}}(\tau)}\right\rangle v^{1/2}\frac{du dv}{v^{2}}.
	\end{align*}
	Now applying Stokes' theorem we obtain
	\begin{align*}
	\Phi(f,w) &= \lim_{T \to \infty}\int_{iT}^{1+iT}\left\langle f_{P \oplus N}(\tau),\overline{\widetilde{\Theta}_{P}(\tau)\otimes \Theta_{N^{-}}(\tau)}\right\rangle d\tau \\
	&\quad-\lim_{T\to \infty}\int_{\mathcal{F}_{T}}\left\langle  \widetilde{\Theta}_{P}(\tau) \otimes \Theta_{N}(\tau), \xi_{2-k}f_{P \oplus N}(\tau)\right\rangle v^{2-k}\frac{du dv}{v^{2}}.
	\end{align*}
	Here we used that $\Theta_{N^{-}}$ is holomorphic. Note that the integral in the first line picks out the Fourier coefficient of index $0$ of the integrand. If we split $f_{P \oplus N}$ and $\widetilde{\Theta}_{P}$ into their holomorphic and non-holomorphic parts, and use that the non-holomorphic parts vanish as $v \to \infty$, we obtain the stated formula. 
\end{proof}

\section{Recurrences for the coefficients of mock theta functions}\label{section applications}

As an application of the different evaluations of the regularized theta lift we derive recurrences for the coefficients of some mock theta functions. We obtain relations for Hurwitz class numbers, Andrews' spt-function and some of Ramanujan's mock theta functions. We only consider lattices of signature $(1,1)$ such that every harmonic Maass form in $H_{0,L}^{\cusp}$ is actually weakly holomorphic, that is, we have $H_{0,L}^{\cusp} = M_{0,L}^{!}$.


\subsection{Isotropic lattices in signature $(1,1)$} We consider the rational quadratic space $\Q^{2}$ with the quadratic form $Q(a,b) = ab$. Let $L$ be an even lattice in $\Q^{2}$. Then $L$ has signature $(1,1)$ and is isotropic.

Let $w \in \Gr(L)$ be a special point and let $y = (y_{1},y_{2}) \in L$ be its primitive generator with $y_{1},y_{2} > 0$. Furthermore, let $y^{\perp} \in L$ be the primitive generator of $w^{\perp}$ whose second coordinate is positive. Then $y^{\perp}$ is a positive multiple of $(-y_{1},y_{2})$. We let $d_{P} = y_{1}y_{2}$ and $d_{N} = y^{\perp}_{1}y^{\perp}_{2}$. Then $P \cong (\Z,d_{P}x^{2})$ and $N \cong (\Z,-d_{N}x^{2})$. In particular, we may identify
\[
\Theta_{P} = \theta_{1/2,d_{P}}, \qquad \Theta_{N^{-}} = \theta_{1/2,d_{N}},
\]
with the unary theta function $\theta_{1/2,d}$ defined in Section~\ref{section unary theta series}. Combining Theorem~\ref{theorem unfolding} and Theorem~\ref{theorem splitting computation} we obtain the following evaluations of the theta lift for isotropic lattices at special points.

\begin{Theorem}\label{theorem main identity isotropic} Suppose that $L$ is as above. Then the evaluation of the regularized theta lift $\Phi(f,w)$ of $f \in M_{0,L}^{!}$ at a special point $w = \R(y_{1},y_{2}) \in \Gr(L)$ is given by
\begin{align*}
\Phi(f,w) &=\frac{4\pi}{\sqrt{y_{1}y_{2}}}\sum_{\substack{\lambda = (\lambda_{1},\lambda_{2}) \in L' \\ \lambda_{1}\lambda_{2} < 0}}a_{f}(\lambda,\lambda_{1}\lambda_{2})\min\left(|\lambda_{1}y_{2}|,|\lambda_{2}y_{1}|\right) \\
& = \sum_{\substack{\alpha_{P}(2d_{P}) \\ \alpha_{N} (2d_{N})\\ \gamma_{P}+\gamma_{N} \in L'}}\sum_{\substack{n \in \Z-\frac{\alpha_{P}^{2}}{4d_{P}}+\frac{\alpha_{N}^{2}}{4d_{N}}}}a_{f}(\gamma_{P}+\gamma_{N},-n)\sum_{\substack{r \in \Z \\ r \equiv \alpha_{N} (2d_{N})}}a_{\widetilde{\theta}_{1/2,d_{P}}}^{+}(\alpha_{P},n-r^{2}/4d_{N}),
\end{align*}
where we wrote $\gamma_{P} = \frac{\alpha_{P}}{2d_{P}}y \in P'$ and $\gamma_{N} = \frac{\alpha_{N}}{2d_{N}}y^{\perp} \in N'$, and $\widetilde{\theta}_{1/2,d_{P}}$ is a $\xi_{3/2}$-preimage of the unary theta function $\theta_{1/2,d_{P}}$. The sum in the first line is finite.
\end{Theorem}

\begin{proof}
	We simplify the formula from Theorem~\ref{theorem unfolding} (or rather its specialization to $n = 1$ given in equation \eqref{eq unfolding 1} in Remark~\ref{remark unfolding}). For $\lambda = (\lambda_{1},\lambda_{2}) \in \Q^{2}$ with $Q(\lambda) = \lambda_{1}\lambda_{2}$ we have
	\[
	Q(\lambda_{w}) = \frac{1}{2}\frac{(\lambda,y)^{2}}{|y|^{2}}, \qquad |Q(\lambda_{w^{\perp}})| = -Q(\lambda)+Q(\lambda_{w}) = -\lambda_{1}\lambda_{2}+\frac{1}{2}\frac{(\lambda,y)^{2}}{|y|^{2}}.
	\]
	Further, we have the explicit formulas
	\begin{align*}
	(\lambda,y) = \lambda_{1}y_{2}+\lambda_{2}y_{1}, \qquad -2\lambda_{1}\lambda_{2}|y|^{2}+(\lambda,y)^{2} = (\lambda_{1}y_{2}-\lambda_{2}y_{1})^{2}, 
	\end{align*}
	so we obtain
	\begin{align*}
	\sqrt{|Q(\lambda_{w^{\perp}})|}-\sqrt{|Q(\lambda_{w})|} = \frac{1}{2\sqrt{y_{1}y_{2}}}\left(|\lambda_{1}y_{2}-\lambda_{2}y_{1}|- |\lambda_{1}y_{2}+\lambda_{2}y_{1}|\right).
	\end{align*}
	Also note that
	\[
	|\lambda_{1}y_{2}-\lambda_{2}y_{1}|- |\lambda_{1}y_{2}+\lambda_{2}y_{1}| = 2\min\left(|\lambda_{1}y_{2}|,|\lambda_{2}y_{1}|\right)
	\]
	since $\lambda_{1}\lambda_{2} < 0$ and $y_{1},y_{2} > 0$. When combined with \eqref{eq unfolding 1}, this yields the formula in the first line in the theorem.
	
	The right-hand side of the formula in the theorem can be obtained from Theorem~\ref{theorem splitting computation} and Remark~\ref{remark constant term}.
\end{proof}

%

By comparing the two finite formulas on the right-hand side of Theorem~\ref{theorem main identity isotropic} we obtain non-trivial recurrence relations for the coefficients of the mock theta function $\widetilde{\theta}_{1/2,d_{P}}^{+}$. We illustrate this in some examples.

\begin{Example}\label{example hurwitz} Let $L$ be the lattice spanned by the vectors $(1,0)$ and $(0,1)$, that is, a hyperbolic plane. Note that $L$ is unimodular, so vector-valued modular forms for $\rho_{L}$ are just scalar-valued modular forms for $\SL_{2}(\Z)$. By \cite{niebur}, Theorem 1, for $m \in \N$ the Maass (or Niebur) Poincar\'e series $F_{\beta,m} \in M_{0,L}^{!}$ is given by the weakly holomorphic modular form $2j_{m}+48\sigma_{1}(m)$, where $j_{m}$ denotes the unique modular function for $\SL_{2}(\Z)$ whose Fourier expansion starts $q^{-m}+O(q)$. 

Let $y = (y_{1},y_{2}) \in \Z^{2}$ with $y_{1},y_{2} > 0$ coprime. Then $y^{\perp} = (-y_{1},y_{2})$ and $d := d_{P} = d_{N} = y_{1}y_{2}$. After some simplification, Theorem~\ref{theorem main identity isotropic} applied to $f = 2j_{m}+48\sigma_{1}(m)$ gives the formula
\begin{align}\label{eq formula hyperbolic plane}
	&\frac{16\pi}{\sqrt{y_{1}y_{2}}}\sum_{\substack{\lambda_{1},\lambda_{2} \in \N \\   \lambda_{1}\lambda_{2} = m}}\min(\lambda_{1}y_{2},\lambda_{2}y_{1})=\sum_{n \in \Z}a_{m}(-n)\sum_{r \in \Z}a_{\widetilde{\theta}_{1/2,d}^{+}}(\sigma_{y_{1}}(r),n-r^{2}/4d),
	\end{align}
	where $\sigma_{c}$ for $c \mid\mid d$ is the Atkin-Lehner involution on $\Z/2d\Z$ defined in Section~\ref{section unary theta series} and $a_{m}(-n)$ denotes the $(-n)$-th coefficient of $2j_{m}+48\sigma_{1}(m)$.
	
	Zagier \cite{zagiereisensteinseries} showed that the $\C[\Z/2\Z]$-valued generating series 
	\[
	\sum_{n = 0}^{\infty}H(n)q^{n/4}\e_{n}, \qquad H(0) = -\frac{1}{12},
	\] 
	of Hurwitz class numbers $H(n)$ is the holomorphic part of a harmonic Maass form of weight $3/2$ for the dual Weil representation associated to the lattice $(\Z,x^{2})$ which maps to $(-1/8\pi)\theta_{1/2,1}$ under $\xi_{3/2}$.
	Applying formula \eqref{eq formula hyperbolic plane} with the special point $y = (1,1)$ with $d = 1$, we recover the classical Hurwitz-Kronecker class number relation
	\begin{align}
\sum_{r \in \Z}H(4m-r^{2}) = 2\sigma_{1}(m) - \sum_{\substack{a,b \in \N \\ ab = m}}\min(a,b)
\end{align}
for $m \in \N$.

More generally, for a square-free positive integer $N$ we consider the $\C[\Z/2N\Z]$-valued generating series
\[
\sum_{r(2N)}\sum_{\substack{n \geq 0 \\ r^{2}\equiv -n (4N)}}H_{r}(n)q^{n/4N}\e_{r}
\]
with the level $N$ Hurwitz class numbers
\[
H_{0}(0) = -\frac{1}{12}\sigma_{1}(N) , \qquad H_{r}(n) = \frac{1}{2}\sum_{\substack{Q \in \mathcal{Q}_{N,-n,r}}/\Gamma_{0}(N)}\frac{1}{|\overline{\Gamma}_{0}(N)_{Q}|} \quad  (n > 0),
\]
where 
\[
\mathcal{Q}_{N,-n,r} = \{[a,b,c] \,: \, a,b,c \in \Z,\, b^{2}-4ac = -n, \, N\mid a, \, b \equiv r (2N)\}.
\]
It was shown in \cite{funke} that it is a mock modular form of weight $3/2$ for the dual Weil representation associated to the lattice $(\Z,Nx^{2})$ whose shadow is given by
\[
-\frac{\sqrt{N}}{8\pi}\sum_{d \mid N}\theta_{1/2,N}^{\sigma_{d}}.
\]
We sum up formula \eqref{eq formula hyperbolic plane} for all special points of discriminant $2N$, which are precisely the points $y = (d,N/d)$ for $d \mid N$, to obtain the relation
\begin{align}
\sum_{r \in \Z}H_{r}(4Nm -r^{2}) = 2\sigma_{1}(N)\sigma_{1}(m) - \sum_{d \mid N}\sum_{\substack{a,b \in \N \\ ab = m}}\min\left(ad,\tfrac{N}{d}b\right).
\end{align}
We remark that there is a vast literature on recurrences for Hurwitz class numbers, see for example \cite{bringmannkanehurwitz, brown, cohen, eichler, kronecker, mertens1,mertens2, williams2, williams}.
\end{Example}

\begin{Example}
	We consider the lattice $L = \Z y + \Z y^{\perp}$ with $y = (1,1)$ and $y^{\perp} = (-1,1)$. In this case we have $L' = \Z y/2+ \Z y^{\perp}/2$. In particular, $L'/L$ is isomorphic to $(\Z/2\Z)^{2}$, so we can write the elements of $L'/L$ in the form $\beta = (\beta_{1},\beta_{2})$ with $\beta_{1},\beta_{2} \in \Z/2\Z$. We apply Theorem~\ref{theorem main identity isotropic} to $f = F_{\beta,m}$ with $\beta = (0,0)$ and $\beta = (1,1)$, respectively. As a $\xi$-preimage of the unary theta function $\theta_{1/2,1}$ we again choose the generating series of Hurwitz class numbers from the last example. The formula in Theorem~\ref{theorem main identity isotropic} involves the constant coefficients $a_{F_{\beta,m}}(\gamma,0)$ of $F_{\beta,m}$, for isotropic $\gamma \in L'/L$. It follows from the residue theorem that this coefficient is given by the negative of the $(\beta,m)$-th coefficient of the vector-valued Eisenstein series $E_{\gamma}$ studied in \cite{brkuss}, and  hence can be computed in terms of divisor sums using the explicit formulas given there. Alltogether, we obtain the formulas
\begin{align}
\sum_{r\equiv 0 (2)}H(4m-r^{2}) = \frac{4}{3}\sigma_{1}(m)-2\sigma_{1}(m/2)+\frac{8}{3}\sigma_{1}(m/4)-\frac{1}{2}\!\!\!\sum_{\substack{a,b \in \Z \\ b^{2}-a^{2} = m}}\!\!\!\min(|a+b|,|a-b|)
\end{align}
and
\begin{align}
	\sum_{r\equiv 1 (2)}H(4m-r^{2}) = \frac{2}{3}\sigma_{1}(m)+2\sigma_{1}(m/2)-\frac{8}{3}\sigma_{1}(m/4)-\frac{1}{2}\!\!\!\!\!\!\!\!\sum_{\substack{a,b \in \Z \\ b^{2}+b-a^{2}-a = m}}\!\!\!\!\!\min(|a+b+1|,|a-b|)
	\end{align}
	for $m \in \N$, where $\sigma_{1}(x) = 0$ if $x$ is not integral. We leave the details of the simplification to the reader. This generalizes some of the formulas found by Brown et al. \cite{brown}, where the authors studied restricted class number sums for $m = p$ being an odd prime. We also refer to \cite{bringmannkanehurwitz, williams2, williams} for similar formulas involving restricted class number sums. 
\end{Example}

\begin{Example}\label{example spt}
	The smallest parts function $\spt(n)$, introduced by Andrews in \cite{andrews}, counts the total number of smallest parts in the partitions of $n$ (with $\spt(0) = 0$). For $n \in \N_{0}$ we let
	\[
	s(n) = \spt(n) +\frac{1}{12}(24 n -1)p(n),
	\]
	with the partition function $p(n)$. In particular, $s(0) = -1/12$. Bringmann \cite{bringmannduke} showed that the generating series
	\[
	q^{-\frac{1}{24}}\sum_{n=0}^{\infty}s(n)q^{n}
	\]
	is the holomorphic part of a harmonic Maass form $F(\tau)$ of weight $\frac{3}{2}$ for $\Gamma_{0}(576)$ with character $\left(\frac{12}{\cdot} \right)$ which satisfies $\xi_{3/2}F = (-\sqrt{6}/4\pi)\eta$. From this one can infer that the function
	\[
	\sum_{r (12)}\left(\frac{12}{r}\right)F(\tau)\e_{r} = F(\tau)(\e_{1}-\e_{5}-\e_{7}+\e_{11})
	\]
	is a harmonic Maass form of weight $3/2$ for the Weil representation associated to $(\Z,6x^{2})$ which maps to 
	\[
	-\frac{\sqrt{6}}{4\pi}\left(\theta_{1/2,6}-\theta_{1/2,6}^{\sigma_{2}}\right)
	\] 
	under $\xi_{3/2}$ (compare \cite{ahlgrenandersen}, Section 3).
	
	We choose the special points $y = (1,6)$ and $y = (2,3)$ in \eqref{eq formula hyperbolic plane} and subtract the resulting equations. We obtain after some simplification the formula
	\begin{align}
	\sum_{r \in \Z}\left( \frac{12}{r}\right)s\left(m-\frac{r^{2}}{24}+\frac{1}{24} \right) = 4\sigma_{1}(m)-2\sum_{d \mid m}\left(\min(6d,m/d)-\min(3d,2m/d) \right)
	\end{align}
	for all $m \in \N$.
\end{Example}

\subsection{Anisotropic lattices in signature $(1,1)$} Let $F = \Q(\sqrt{D})$ with a non-square discriminant $D > 0$. It is a rational quadratic space with the quadratic form given by the norm 
\[
Q(\nu) = N_{F/\Q}(\nu) = \nu\nu'.
\]
The associated bilinear form is given by the trace 
\[
(\nu,\mu) = \tr_{F/\Q}(\nu\mu') = \nu\mu'+\nu'\mu.
\]
Let $L \subset F$ be an even lattice. Then $L$ has signature $(1,1)$ and is anisotropic.

We let $w \in \Gr(L)$ be a special point and we let $y \in L$ be the primitive generator of $w$ with $y,y' > 0$. Furthermore, we let $y^{\perp}$ be the primitive generator of $w^{\perp}$ with $y^{\perp} < 0$ and $y^{\perp'} > 0$. Let $d_{P} = yy'$ and $d_{N} = y^{\perp}y^{\perp'}$. Then $P \cong (Z,d_{P}x^{2})$ and $N \cong (\Z, -d_{N}x^{2})$.
%

Using similar arguments as in the proof of Theorem~\ref{theorem main identity isotropic} we obtain the following formula.

\begin{Theorem}\label{theorem main identity anisotropic} Suppose that $L$ is as above. Then the evaluation of the regularized theta lift $\Phi(f,w)$ of $f \in M_{0,L}^{!}$ at a special point $w = \R y \in \Gr(L)$ is given by
\begin{align*}
\Phi(f,w) &=\frac{4\pi}{\sqrt{yy'}}\sum_{\substack{\lambda \in L' \\ \lambda\lambda' <0}}a_{f}(\lambda,\lambda\lambda')\min(|\lambda y'|,|\lambda'y|)\\
& = \sum_{\substack{\alpha_{P}(2d_{P}) \\ \alpha_{N} (2d_{N})\\ \gamma_{P}+\gamma_{N} \in L'}}\sum_{\substack{n \in \Z-\frac{\alpha_{P}^{2}}{4d_{P}}+\frac{\alpha_{N}^{2}}{4d_{N}}}}a_{f}(\gamma_{P}+\gamma_{N},-n)\sum_{\substack{r \in \Z \\ r \equiv \alpha_{N} (2d_{N})}}a_{\widetilde{\theta}_{1/2,d_{P}}}^{+}(\alpha_{P},n-r^{2}/4d_{N}),
\end{align*}
where we wrote $\gamma_{P} = \frac{\alpha_{P}}{2d_{P}}y\in P'$ and $\gamma_{N} = \frac{\alpha_{N}}{2d_{N}}y^{\perp} \in N'$, and $\widetilde{\theta}_{1/2,d_{P}}$ is a $\xi_{3/2}$-preimage of the unary theta function $\theta_{1/2,d_{P}}$.
\end{Theorem}

\begin{Remark}\label{remark anisotropic} Note that the sum in the first line of the theorem is not finite due to the existence of units. However, according to the discussion after Theorem 9 in Section 4 of \cite{bruinierbundschuh}, it can be rewritten as a finite sum as follows. First, we can split the sum as
\[
\sum_{\substack{\lambda \in L' \\ \lambda\lambda' <0}}a_{f}(\lambda,\lambda\lambda')\min(|\lambda y'|,|\lambda'y|) = \sum_{\beta \in L'/L}\sum_{m < 0}a_{f}(\beta,m)\sum_{\substack{\lambda \in L+\beta \\ \lambda\lambda' = m}}\min(|\lambda y'|,|\lambda'y|),
\]
where the first two sums are finite. The set of units of norm one acting on $L \pm \beta = (L + \beta) \cup (L-\beta)$ is of the form $\{\pm 1\} \times \{\varepsilon_{\beta}^{n}: n \in \Z\}$ for some unit $\varepsilon_{\beta} > 1$. Further, the set 
\[
R(m) = \{\lambda \in L\pm \beta: \, \lambda > 0, \, \lambda\lambda' = m, \, (\lambda ,y) < 0, \, (\lambda \varepsilon_{\beta},y) \geq 0\}
\]
is finite, and by Dirichlet's unit theorem the set of all $\lambda \in L \pm \beta$ with $\lambda > 0$ and $\lambda\lambda' = m$ is given by
\[
\{\pm \lambda \varepsilon_{\beta}^{n}: \lambda \in R(m), \, n \in \Z\}.
\]
For $\lambda \in R(m)$ and $n \in \Z$ we have
\[
\min(|\lambda \varepsilon_{\beta}^{n}y'|,|\lambda'\varepsilon_{\beta}'^{n}y|) = \begin{cases}
\lambda \varepsilon_{\beta}^{n}y', & \text{if } n \leq 0,\\
-\lambda'\varepsilon_{\beta}'^{n}y, & \text{if }n > 0.
\end{cases}
\]
Thus we obtain
\begin{align*}
\sum_{\substack{\lambda \in L\pm\beta \\ \lambda > 0 \\ \lambda\lambda' = m}}\min(|\lambda y'|,|\lambda'y|) &= 2\sum_{\lambda \in R(m)}\left(\lambda y'\sum_{n \leq 0}\varepsilon_{\beta}^{n}-\lambda'y\sum_{n > 0}\varepsilon_{\beta}'^{n} \right) \\
&= 2\sum_{\lambda \in R(m)}\left( \frac{\lambda y'}{1-\varepsilon_{\beta}'}+\frac{\lambda'y}{1-\varepsilon_{\beta}} \right) = 2\sum_{\lambda \in R(m)}\tr_{F/\Q}\left(\frac{\lambda y'}{1-\varepsilon_{\beta}'}\right).
\end{align*}
\end{Remark}

\begin{Example} Let $L = \mathcal{O}_{F}$ be the ring of integers of $F = \Q(\sqrt{D})$. The dual lattice of $L$ is given by the inverse different $\mathfrak{d}^{-1} = (1/\sqrt{D})\mathcal{O}_{F}$ and the discriminant group of $L$ has level $D$.

Let $\beta = 0$ for simplicity. Then $\varepsilon_{0}$ is the smallest unit $> 1$ of norm $1$ (which is either the fundamental unit of $F$ or its square). The coefficient of index $(0,0)$ of the Niebur Poincar\'e series $F_{0,m}$ is given by the negative of the coefficient of index $(0,m)$ of the weight $2$ Eisenstein series $E = 2\e_{0} + O(q)$ for the dual Weil representation, compare \cite{brhabil}, Proposition~1.14. By results of Zhang \cite{zhang}, holomorphic vector-valued modular forms of weight $k$ for the Weil representation associated with $\mathcal{O}_{F}$ can be identified with a certain space $M_{k}^{\epsilon}(D,\chi_{D})$ of scalar-valued modular forms for $\Gamma_{0}(D)$ with character $\chi_{D} = \left( \frac{D}{\cdot}\right)$. In particular, by \cite{zhang}, Section~4, the coefficient of index $(0,m)$ of the vector-valued Eisenstein series $E$ is equal to the $Dm$-th coefficient of the scalar-valued weight $2$ Eisenstein series 
\[
E_{2}^{(D)}(\tau) = 2+\frac{4}{L(-1,\chi_{D})}\sum_{n=1}^{\infty}\sigma_{1}^{(D)}(n) q^{n} \in M_{2}^{\epsilon}(D,\chi_{D}),
\]
with the usual Dirichlet $L$-function $L(s,\chi_{D})$ and the twisted divisor sum
\[
\sigma_{1}^{(D)}(n) = \sum_{d \mid n}\sum_{m \mid \mid D}\chi_{m}(n/d)\chi_{m}'(d)d,
\]
where $\chi_{m}$ denotes the $m$-part of $\chi_{D}$ and we write $\chi_{D} = \chi_{m}\chi_{m}'$.

We choose the special point $y = \varepsilon_{0}$, such that $d_{P} = yy' = 1$. Its orthogonal complement is spanned by $y^{\perp} = \sqrt{D}\varepsilon_{0}$. Thus $d_{N} = y^{\perp}y^{\perp'} = D$. 
The space of cusp forms of weight $2$ for $\rho_{L^{-}}$ is trivial for a fundamental discriminant $D$ if and only if $D \leq 21$ (see Lemma~2 in \cite{williams}). This implies that the Maass-Poincar\'e series $F_{0,m}$ is weakly holomorphic for every $m$. We apply Theorem~\ref{theorem main identity anisotropic} for $f = F_{0,m}$ with the generating series of Hurwitz class numbers and obtain after some simplification the formula
\begin{align}
\sum_{r \in \Z}H(4m-Dr^{2}) = -\frac{1}{6L(-1,\chi_{D})}\sigma_{1}^{(D)}(Dm)-\sum_{\lambda \in R(m)}\tr_{F/\Q}\left(\frac{\lambda}{\varepsilon_{0}-1}\right)
\end{align}
for $D = 5,8,12,13,17,21$. Note that $L(-1,\chi_{D})$ is a rational number. These relations were first discovered by Williams \cite{williams} by comparing the Fourier coefficients of vector-valued Hirzebruch-Zagier series and Eisenstein series for the Weil representation.
\end{Example}

\subsection{Ramanujan's mock theta functions}\label{section mock theta functions}
By using a different theta function, the methods of this work can also be applied to obtain relations between the coefficients of $\xi_{1/2}$-preimages of unary theta series of weight $3/2$, e.g., of the coefficients of the classical mock theta functions of Ramanujan. We indicate the necessary steps but omit some details. Let $L$ be an even lattice of signature $(1,1)$. We fix an isotropic vector $\ell \in V$ and define the modified theta function
\[
\Theta_{L}^{*}(\tau,z) = v^{3/2}\sum_{\lambda \in L'}\left(\lambda,\frac{\ell_{z}}{|\ell_{z}|}\right)\left(\lambda,\frac{\ell_{z^{\perp}}}{|\ell_{z^{\perp}}|}\right)e(Q(\lambda_{z})\tau+Q(\lambda_{z^{\perp}})\overline{\tau})\e_{\lambda + L}.
\]
It has weight $0$ in $\tau$ for $\rho_{L}$, and at a special point $w \in \Gr(L)$ it splits as a tensor product of unary theta functions of weight $3/2$, similarly as in \eqref{splitting siegel theta}. The corresponding theta lift $\Phi^{*}(f,z)$ is defined analogously to \eqref{eq definition theta lift}. Theorem~\ref{theorem unfolding} and Theorem~\ref{theorem splitting computation} have similar looking analogs whose proofs require only minor modifications. In particular, if $L$ is a hyperbolic plane, we obtain the formula
\begin{align}\label{eq formula hyperbolic plane modified}
\begin{split}
&\frac{8}{\sqrt{y_{1}y_{2}}}\sum_{\substack{\lambda_{1},\lambda_{2} \in \N \\ \lambda_{1}\lambda_{2} = m}}\sgn(\lambda_{2}y_{1}-\lambda_{1}y_{2})\min\left(\lambda_{1}y_{2},\lambda_{2}y_{1}\right)  \\
&= \sum_{n \in \Z}a_{m}(-n)\sum_{r \in \Z}r  a_{\widetilde{\theta}_{3/2,d_{P}}}^{+}(\sigma_{y_{1}}(r),n-r^{2}/4d),
\end{split}
\end{align}
with the same notation as in \eqref{eq formula hyperbolic plane}.

For example, let $f(q)$ and $\omega(q)$ be Ramanujan's mock theta functions of order $3$. The fundamental work of Zwegers \cite{zwegerspaper,zwegers} showed that the $\C[\Z/12\Z]$-valued function
\begin{align*}
F(\tau) &= q^{-1/24}f(q)(\e_{1}-\e_{5}+\e_{7}-\e_{11}) \\
&\quad + 2q^{1/3}(\omega(q^{1/2})-\omega(-q^{1/2}))(-\e_{2}+\e_{10}) \\
&\quad + 2q^{1/3}(\omega(q^{1/2})+\omega(-q^{1/2}))(-\e_{4}+\e_{8})
\end{align*}
can be completed to a harmonic Maass form of weight $1/2$ for the Weil representation of the lattice $(\Z,6x^{2})$ which maps to a linear combination of unary theta functions of weight $3/2$
under $\xi_{1/2}$ (compare \cite{bruinieronoheegnerdivisors}, Section 8.2). We apply formula \eqref{eq formula hyperbolic plane modified} with the special points $y =(1,6)$ and $y = (2,3)$ and sum up the resulting equations. After some simplification we obtain the formula
\begin{align}
\begin{split}
&\sum_{\substack{r \in \Z \\ r^{2} \equiv 1 (24)}}\!\!\!\left(\frac{-12}{r}\right)ra_{f}\left(m-\frac{r^{2}}{24}+\frac{1}{24}\right) + 4\!\!\!\sum_{\substack{r \in \Z \\ r^{2}\equiv 4 (12)}}\!\!\! e\left(\frac{r^{2}-4}{24}\right)\left(\frac{-3}{r}\right)ra_{\omega}\left(2m-\frac{r^{2}}{12}-\frac{2}{3}\right)  \\
&=-48\sigma_{1}(m) +  \sum_{d \mid m}\big(\sgn(d-6\tfrac{m}{d})\min(d,6\tfrac{m}{d})+\sgn(2d-3\tfrac{m}{d})\min(2d,3\tfrac{m}{d}) \big)
\end{split}
\end{align}
for $m \in \N$. We remark that more refined relations for the coefficients of Ramanujan's $f(q)$ and $w(q)$ were given by Imamoglu, Raum and Richter \cite{irr} and Chan, Mao and Osburn \cite{chanmaoosburn}. Their methods work for other mock theta functions, too.
Klein and Kupka \cite{kleinkupka} explicitly worked out the completions to vector-valued harmonic Maass forms for the Weil representation for many of Ramanujan's classical mock theta functions. Hence the methods of the present work can be used to derive recursions for their coefficients as well.

\subsection{A lattice of signature $(1,2)$}\label{section signature 12} Here we discuss the example from the introduction. The lattice $\Z^{3}$ with the quadratic form $Q(a,b,c) = -b^{2}+ac$ can also be realized as the set
\[
L = \left\{\lambda = \begin{pmatrix}b & c \\ -a & -b \end{pmatrix}: a,b,c \in \Z\right\}
\]
of traceless integral $2$ by $2$ matrices with the quadratic form $Q(\lambda) = \det(\lambda)$ and the associated bilinear form $(\lambda,\mu) = -\tr(\lambda\mu)$. Its dual lattice is given by
\[
L' = \left\{\lambda = \begin{pmatrix}b /2& c \\ -a & -b/2 \end{pmatrix}: a,b,c \in \Z\right\}.
\]
In particular, $L$ has level $4$ and $L'/L$ can be identified with $\Z/2\Z$. 

By the results of \cite{eichlerzagier}, Section 5, vector-valued modular forms of weight $k$ for the Weil representation $\rho_{L}$ can be identified with scalar-valued modular forms of weight $k$ for $\Gamma_{0}(4)$ satisfying the Kohnen plus space condition via the map
\[
f_{0}(\tau)\e_{0} + f_{1}(\tau)\e_{1} \mapsto f_{0}(4\tau) + f_{1}(4\tau).
\]
In particular, we can use scalar-valued modular forms as inputs for the theta lift. 
%
Under the above map the vector-valued Maass Poincar\'e series $F_{\beta,m}$ of weight $-1/2$ corresponds to the scalar-valued weakly holomorphic modular form $2f_{D}$ with $D=4m$, where $f_{D} = q^{-D} + O(1)$ is the modular form described in the introduction. Note that $f_{D}$ is weakly holomorphic since the Kohnen plus space of cusp forms of weight $5/2$ for $\Gamma_{0}(4)$ is trivial, which in turn follows from the fact that it is isomorphic to the space of cusp forms of weight $6$ for $\SL_{2}(\Z)$ under the Shimura correspondence.

The modular group $\Gamma = \SL_{2}(\Z)$ acts on $L$ as isometries by $\gamma.X = \gamma X \gamma^{-1}$. Furthermore, the Grassmannian $\Gr(L)$ of positive lines in $V= L \otimes \R$ can be identified with the upper half-plane $\H$ by mapping $z \in \H$ to the positive line generated by
\[
\lambda(z) = \frac{1}{\sqrt{2}y}\begin{pmatrix}-x & x^{2}+y^{2} \\ -1 & x \end{pmatrix}.
\]
Note that $(\lambda(z),\lambda(z)) =1$. Under this identification the action of $\Gamma$ on $\Gr(L)$ corresponds to the action by fractional linear transformations on $\H$. A short compuation shows that for $\lambda = \left(\begin{smallmatrix}b/2 & c \\ -a & -b/2 \end{smallmatrix}\right) \in L'$ and $z = x+iy \in \H \cong \Gr(L)$ we have
\begin{align*}
Q(\lambda_{z}) &= Q((\lambda,\lambda(z))\lambda(z))= \frac{1}{2}(\lambda,\lambda(z))^{2} = \frac{1}{4y^{2}}(a|z|^{2}+bx + c)^{2}, \\
Q(\lambda_{z^{\perp}}) &= Q(\lambda) - Q(\lambda_{z}) = -\frac{1}{4y^{2}}|az^{2}+bz+c|^{2}.
\end{align*}
We can now explain how the theorems in the introduction follow from the results from the body of the paper.

\begin{proof}[Proof of Theorem~\ref{theorem unfolding introduction}]
Using the above evaluations of $Q(\lambda_{z})$ and $Q(\lambda_{z^{\perp}})$, from Theorem~\ref{theorem unfolding} and Remark~\ref{remark unfolding} we obtain the series representation
\[
	\Phi(f_{D},z) = 4\sum_{[a,b,c]\in \mathcal{Q}_{D}}\left(\sqrt{D}-\frac{1}{y}|a|z|^{2}+bx+c|\arcsin\left(\frac{y\sqrt{D}}{|az^{2}+bz+c|} \right) \right),
	\]
which yields Theorem~\ref{theorem unfolding introduction} in the introduction.
\end{proof}

\begin{proof}[Proof of Theorem~\ref{theorem fourier expansion introduction}]
We consider the isotropic vectors
\[
\ell = \begin{pmatrix}0 & 1 \\ 0 & 0 \end{pmatrix}, \qquad \ell' = \begin{pmatrix}0 & 0\\ -1 & 0 \end{pmatrix} ,
\]
in $L$ with $(\ell,\ell') = 1$ and the one-dimensional negative definite lattice
\[
K = L \cap \ell^{\perp} \cap \ell'^{\perp} = \left\{\lambda = \begin{pmatrix}b & 0 \\ 0 & -b \end{pmatrix} : b \in \Z\right\} \cong (\Z,-b^{2}).
\]
Clearly, we can identify hyperbolic $2$-space $\H^{2}$ with the complex upper half-plane $\H$. Using Theorem~\ref{theorem fourier expansion} we now easily obtain Theorem~\ref{theorem fourier expansion introduction} from the introduction.
\end{proof}

\begin{proof}[Proof of Theorem~\ref{theorem splitting introduction}]
The CM point $i \in \H$ corresponds to the special point $w \in \Gr(L)$ generated by the matrix $\left(\begin{smallmatrix}0 & 1 \\ -1 & 0 \end{smallmatrix} \right)$. The corresponding positive and negative definite lattices are given by
\[
P = L \cap w = \left\{x\begin{pmatrix}0 & 1 \\ -1 & 0 \end{pmatrix}: x\in \Z\right\}, \qquad N = L \cap w^{\perp} = \left\{\begin{pmatrix}b & -a & \\ -a &-b \end{pmatrix}: a,b \in \Z \right\}.
\]
In particular, we see that $P$ is isomorphic to $(\Z,x^{2})$ and $N$ is isomorphic to $(\Z^{2},-a^{2}-b^{2})$. If we now apply Theorem~\ref{theorem splitting computation} together with Remark~\ref{remark constant term} to Zagier's $\C[\Z/2\Z]$-valued generating of Hurwitz class numbers discussed in Example~\ref{example hurwitz}, we obtain Theorem~\ref{theorem splitting introduction}.
\end{proof}
%
%
%
%
%

\bibliography{references}{}

\begin{thebibliography}{10}

\bibitem{ahlgrenandersenspt}
Scott Ahlgren and Nickolas Andersen.
\newblock Euler-like recurrences for smallest parts functions.
\newblock {\em Ramanujan J.}, 36:237--248, 2015.

\bibitem{ahlgrenandersen}
Scott Ahlgren and Nickolas Andersen.
\newblock Algebraic and transcendental formulas for the smallest parts
  function.
\newblock {\em Adv. Math.}, 289:411--437, 2016.

\bibitem{andrews}
George~E. Andrews.
\newblock The number of smallest parts in the partitions of $n$.
\newblock {\em J. Reine Angew. Math.}, 624:133--142, 2008.

\bibitem{borcherds}
Richard~E. Borcherds.
\newblock Automorphic forms with singularities on {G}rassmannians.
\newblock {\em Invent. Math.}, 132(3):491--562, 1998.

\bibitem{bringmannduke}
Kathrin Bringmann.
\newblock On the explicit construction of higher deformations of partition
  statistics.
\newblock {\em Duke Math. J.}, 144:195--233, 2008.

\bibitem{bringmannkanehurwitz}
Kathrin Bringmann and Ben Kane.
\newblock Sums of class numbers and mixed mock modular forms.
\newblock {\em Math. Proc. Camb. Philos. Soc.}, 167:321--333, 2019.

\bibitem{bringmannkanekohnen}
Kathrin Bringmann, Ben Kane, and Winfried Kohnen.
\newblock Locally harmonic {M}aass forms and the kernel of the {S}hintani lift.
\newblock {\em Int. Math. Res. Not. IMRN}, 2015:3185--3224, 2015.

\bibitem{brown}
Brittany Brown, Neil~J. Calkin, Timothy~B. Flowers, Kevin James, Ethan Smith,
  and Amy Stout.
\newblock Elliptic curves, modular forms, and sums of {H}urwitz class numbers.
\newblock {\em J. Number Theory}, 128:1847--1863, 2008.

\bibitem{brhabil}
Jan~H. Bruinier.
\newblock {\em Borcherds products on {O}(2, {$l$}) and {C}hern classes of
  {H}eegner divisors}, volume 1780 of {\em Lecture Notes in Mathematics}.
\newblock Springer-Verlag, Berlin, 2002.

\bibitem{bruinierconverse}
Jan~H. Bruinier.
\newblock On the converse theorem for {B}orcherds products.
\newblock {\em J. Algebra}, 397:315--342, 2014.

\bibitem{bruinierbundschuh}
Jan~H. Bruinier and Michael Bundschuh.
\newblock {O}n {B}orcherds products associated with lattices of prime
  discriminant.
\newblock {\em Ramanujan J.}, 7:49--61, 2003.

\bibitem{bruinierfunke04}
Jan~H. Bruinier and Jens Funke.
\newblock On two geometric theta lifts.
\newblock {\em Duke Math. J.}, 125(1):45--90, 2004.

\bibitem{brkuss}
Jan~H. Bruinier and Michael Kuss.
\newblock Eisenstein series attached to lattices and modular forms on
  orthogonal groups.
\newblock {\em Manuscripta Math.}, 106(4):443--459, 2001.

\bibitem{bruinieronoheegnerdivisors}
Jan~H. Bruinier and Ken Ono.
\newblock Heegner divisors, {$L$}-functions and harmonic weak {M}aass forms.
\newblock {\em Ann. of Math. (2)}, 172(3):2135--2181, 2010.

\bibitem{bruinieryang}
Jan~H. Bruinier and Tonghai Yang.
\newblock Faltings heights of {CM} cycles and derivatives of {$L$}- functions.
\newblock {\em Invent. Math.}, 177(3):631--681, 2009.

\bibitem{bruinierzemel}
Jan~H. Bruinier and Shaul Zemel.
\newblock Special cycles on toroidal compactifications of orthogonal {S}himura
  varieties.
\newblock {\em preprint}, 2019.

\bibitem{chanmaoosburn}
Song~Heng Chan, Renrong Mao, and Robert Osburn.
\newblock On recursions for coefficients of mock theta functions.
\newblock {\em Res. Number Theory}, 1(29):1--18, 2015.

\bibitem{cohen}
Henri Cohen.
\newblock Sums involving the values at negative integers of {$L$}-functions of
  quadratic characters.
\newblock {\em Math. Ann.}, 217:271--285, 1975.

\bibitem{ehlenduke}
Stephan Ehlen.
\newblock {CM} values of regularized theta lifts and harmonic {M}aass forms of
  weight one.
\newblock {\em Duke Math. J.}, 166(13):2447--2519, 2017.

\bibitem{eichler}
Martin Eichler.
\newblock On the class number of imaginary quadratic fields and the sums of
  divisors of natural numbers.
\newblock {\em J. Ind. Math. Soc.}, 15:153--180, 1955.

\bibitem{eichlerzagier}
Martin Eichler and Don Zagier.
\newblock {\em The theory of {J}acobi forms}, volume~55 of {\em Progress in
  Mathematics}.
\newblock Birkh\"auser Boston Inc., Boston, MA, 1985.

\bibitem{funke}
Jens Funke.
\newblock Heegner divisors and nonholomorphic modular forms.
\newblock {\em Compositio Math.}, 133(3):289--321, 2002.

\bibitem{gz}
Benedict~H. Gross and Don Zagier.
\newblock Heegner points and derivatives of {$L$}-series.
\newblock {\em Inventiones mathematicae}, 84(2):225--320, 1986.

\bibitem{hirzebruchzagier}
Friedrich Hirzebruch and Don Zagier.
\newblock Intersection numbers of curves on {H}ilbert modular surfaces and
  modular forms of {N}ebentypus.
\newblock {\em Invent. Math.}, 36:57--113, 1976.

\bibitem{hurwitz2}
Adolf Hurwitz.
\newblock {\"U}ber {R}elationen zwischen {K}lassenzahlen bin\"arer
  quadratischer {F}ormen von negativer {D}eterminante.
\newblock {\em Math. Ann.}, 25:157--196, 1885.

\bibitem{irr}
{\"O}zlem Imamoglu, Martin Raum, and Olav~K. Richter.
\newblock Holomorphic projections and {R}amanujan's mock theta functions.
\newblock {\em Proc. Natl. Acad. Sci. U.S.A.}, 111(11):3961--3967, 2014.

\bibitem{kleinkupka}
David Klein and Jennifer Kupka.
\newblock Completions and algebraic formulas for the coefficients of
  {R}amanujan's mock theta functions.
\newblock {\em preprint}, 2019.

\bibitem{kz}
Winfried Kohnen and Don Zagier.
\newblock Values of {$L$}-series of modular forms at the center of the critical
  strip.
\newblock {\em Invent. Math.}, 64(2):175--198, 1981.

\bibitem{kronecker}
Leopold Kronecker.
\newblock {\"U}ber die {A}nzahl der verschiedenen {K}lassen quadratischer
  {F}ormen von negativer {D}eterminante.
\newblock {\em J. Reine Angew. Math.}, 57:248--255, 1860.

\bibitem{mertens1}
Michael Mertens.
\newblock Mock modular forms and class number relations.
\newblock {\em Res. Math. Sci}, 1(6):1--16, 2014.

\bibitem{mertens2}
Michael Mertens.
\newblock {E}ichler-{S}elberg type identities for mixed mock modular forms.
\newblock {\em Adv. Math.}, 301:359--382, 2016.

\bibitem{niebur}
Douglas Niebur.
\newblock A class of nonanalytic automorphic functions.
\newblock {\em Nagoya Math. J.}, 52:133--145, 1973.

\bibitem{scheithauer}
Nils Scheithauer.
\newblock Some constructions of modular forms for the {W}eil representation of
  {$\mathrm{SL}_2(\mathbb{Z})$}.
\newblock {\em Nagoya Math. J.}, 220:1--43, 2015.

\bibitem{schofer}
Jarad Schofer.
\newblock {Borcherds forms and generalizations of singular moduli}.
\newblock {\em J. Reine Angew. Math.}, 629(629):1--36, 2009.

\bibitem{williams2}
Brandon Williams.
\newblock Overpartition {$M2$}-rank differences, class number relations, and
  vector-valued mock {E}isenstein series.
\newblock {\em Acta Arith.}, 189(4):347--365, 2019.

\bibitem{williams}
Brandon Williams.
\newblock {V}ector-valued {H}irzebruch-{Z}agier series and class number sums.
\newblock {\em Res. Math. Sci.}, 5(25):1--13, 2019.

\bibitem{zagiereisensteinseries}
Don Zagier.
\newblock Nombres de classes et formes modulaires de poids $3/2$.
\newblock {\em C.R. Acad. Sci. Paris (A)}, 281, 1975.

\bibitem{zhang}
Yichao Zhang.
\newblock An isomorphism between scalar-valued modular forms and modular forms
  for {W}eil representations.
\newblock {\em Ramanujan J.}, 37:181--201, 2015.

\bibitem{zwegerspaper}
Sander~P. Zwegers.
\newblock Mock $\vartheta$-functions and real analytic modular forms.
\newblock {\em q-series with applications to combinatorics, number theory, and
  physics (Ed. B. C. Berndt and K. Ono), Contemp. Math.}, 291:269--277, 2001.

\bibitem{zwegers}
Sander~P. Zwegers.
\newblock Mock theta functions.
\newblock {\em Utrecht PhD thesis}, 2002.

\end{thebibliography}
\bibliographystyle{plain}

\end{document}